\documentclass[11pt]{article}

\usepackage{paralist, enumerate, amsmath, amsthm, amsfonts, amssymb, xy, tabmac,
  color}
\usepackage[margin=1in]{geometry} 
\usepackage{hyperref}

\input xy
\xyoption{all}

\numberwithin{equation}{section}
\newtheorem{theorem}[equation]{Theorem}
\newtheorem{proposition}[equation]{Proposition}
\newtheorem{lemma}[equation]{Lemma}
\newtheorem{example}[equation]{Example}
\newtheorem{rmk}[equation]{Remark}
\newtheorem{corollary}[equation]{Corollary}

\newtheorem{problem}[equation]{Problem}

\newenvironment{eg}{\begin{example}\rm}{\end{example}}
\newenvironment{remark}{\begin{rmk}\rm}{\end{rmk}}

\newcommand{\arxiv}[1]{\href{http://arxiv.org/abs/#1}{{\tt arXiv:#1}}}

\newcommand{\N}{\mathbf{N}}
\newcommand{\Q}{\mathbf{Q}}

\newcommand{\Z}{\mathbf{Z}}

\renewcommand{\phi}{\varphi}
\renewcommand{\emptyset}{\varnothing}
\newcommand{\eps}{\varepsilon}

\renewcommand{\tilde}[1]{\widetilde{#1}}
\newcommand{\ol}[1]{\overline{#1}}
\newcommand{\ul}[1]{\underline{#1}}
\newcommand{\DS}{\displaystyle}
\newcommand{\Hs}{\mathrm{H}}

\makeatletter
\def\Ddots{\mathinner{\mkern1mu\raise\p@
\vbox{\kern7\p@\hbox{.}}\mkern2mu
\raise4\p@\hbox{.}\mkern2mu\raise7\p@\hbox{.}\mkern1mu}}
\makeatother

\DeclareMathOperator{\coker}{coker} 
\newcommand{\GL}{\mathbf{GL}}
\newcommand{\SL}{\mathbf{SL}}
\newcommand{\Sp}{\mathbf{Sp}}
\newcommand{\SO}{\mathbf{SO}}

\newcommand{\Sym}{\mathrm{Sym}}

\newcommand{\Sc}{\mathbf{S}}
\DeclareMathOperator{\Tor}{Tor}
\renewcommand{\SS}{\mathfrak{S}}
\renewcommand{\P}{\mathbf{P}}

\newcommand{\B}{\mathbf{B}}
\newcommand{\HS}{\mathrm{HS}}
\newcommand{\VS}{\mathrm{VS}}
\newcommand{\Ad}{\mathrm{Ad}}
\newcommand{\SB}{S\mathbf{B}}
\newcommand{\SQ}{S\mathbf{Q}}
\newcommand{\ch}{\operatorname{ch}}

\DeclareMathOperator{\Spec}{Spec}
\newcommand{\BB}{\mathcal{B}}
\newcommand{\EE}{\mathcal{E}}

\newcommand{\OO}{\mathcal{O}}
\newcommand{\QQ}{\mathcal{Q}}
\newcommand{\RR}{\mathcal{R}}
\newcommand{\TT}{\mathcal{T}}
\DeclareMathOperator{\grade}{grade}

\title{Pieri resolutions for classical groups} 

\author{Steven V Sam \and Jerzy Weyman}
\date{May 19, 2012}

\begin{document}

\maketitle

\begin{center}Dedicated to Corrado De Concini on the occasion of his
  60th birthday.\end{center}

\begin{abstract} 
  We generalize the constructions of Eisenbud, Fl\o ystad, and Weyman
  for equivariant minimal free resolutions over the general linear
  group, and we construct equivariant resolutions over the orthogonal
  and symplectic groups. We also conjecture and provide some partial
  results for the existence of an equivariant analogue of
  Boij--S\"oderberg decompositions for Betti tables, which were proven
  to exist in the non-equivariant setting by Eisenbud and
  Schreyer. Many examples are given.
\end{abstract}


\tableofcontents

\section*{Introduction.}

In this paper we define new families of equivariant free
resolutions. These extend the equivariant resolutions defined in
\cite{efw} over the general linear group, which gave the first proof
in characteristic 0 of the existence of the ``pure'' resolutions
conjectured by Boij and S\"oderberg \cite{bsconj}. The proof of
acyclicity of the pure resolutions in \cite{efw} was indirect and used
the Borel--Weil--Bott theorem. Here we give a more direct proof based
on an explicit description of the differentials due to Olver
\cite{olver}\footnote[0]{In a previous version of this article, we overlooked that Olver did this in \cite[\S 8]{olver} before the appearance of \cite{efw}!}. We also extend the constructions of \cite{efw} to more
general resolutions, and give constructions for other groups. We also
provide some evidence for an equivariant form of the Boij--S\"oderberg
conjectures that would have striking consequences for Schur functions.

Let us describe the setup more precisely and give an overview of the
paper. 

We work over a field $K$ of characteristic zero, and let $V$ be an
$n$-dimensional vector space over $K$. We set $A = \Sym(V) \cong
K[x_1, \dots, x_n]$ to be a polynomial ring in $n$ variables, and we
consider the general linear group $\GL(V)$ as an algebraic group over
$K$. In this paper, all modules are assumed to be graded. We consider
finitely generated equivariant $A$-modules $M$. This means that one
has an (algebraic) action of $\GL(V)$ on $M$ denoted $g.m$ for $g \in
\GL(V)$ and $m \in M$, and an action of $A$ on $M$ denoted $p\cdot m$
for $p \in A$ such that the identity $g.(p\cdot m) = (g.p) \cdot
(g.m)$ holds, where $g.p$ denotes the canonical action of $\GL(V)$ on
$\Sym(V)$. Note that since $\GL(V)$ (and other classical groups used
below) are linearly reductive, the category of
$\GL(V)$-representations is semisimple, so every graded equivariant
$A$-module $M$ has a minimal graded equivariant resolution whose terms
are direct sums of free modules of type $A \otimes \Sc_\lambda V$ where
$\Sc_\lambda V$ denotes the irreducible representation of $\GL(V)$ of
highest weight $\lambda$. The above statements remain true when we
replace $\GL(V)$ by $\SO(V)$ or $\Sp(V)$ whenever $V$ has a
nondegenerate orthogonal or symplectic form.

Section~\ref{equivariantressection} extends the results of
\cite{efw}. In that paper, pure free resolutions are constructed in
characteristic 0 using representation theory. One of the main
constructions was the minimal free resolution of the cokernel of a
nonzero map of the form
\begin{align} \label{pierimap} \phi(\alpha, \beta) \colon A \otimes
  \Sc_\beta V \to A \otimes \Sc_\alpha V
\end{align}
where $\alpha$ and $\beta$ are partitions satisfying $\alpha_1 <
\beta_1$, and $\alpha_i = \beta_i$ for $i>1$. In
Section~\ref{purefreesection}, we give a simpler proof of the
correctness of the terms of the minimal resolution of $\coker
\phi(\alpha, \beta)$. Then in Section~\ref{pieriresolutionsection}, we
extend the construction of minimal free resolutions in
Theorem~\ref{r=1} by removing the restrictions on $\alpha$ and
$\beta$. In particular, we only assume that $\Sc_\beta V$ is a
subrepresentation of $A \otimes \Sc_\alpha V$, so that a nonzero
equivariant map of the form \eqref{pierimap} exists. We call such maps
{\bf Pieri maps}. The decomposition of the modules in the resolution
in terms of $\GL(V)$ representations can be described purely
combinatorially in terms of partitions. We also present a simple
combinatorial algorithm for writing down a free resolution of the
cokernel of a map of the form
\begin{align} \label{pierimap2}
\phi(\alpha; \beta^1, \dots, \beta^r) \colon \bigoplus_{i=1}^r A
\otimes \Sc_{\beta^i} V \to A \otimes \Sc_\alpha V
\end{align}
in Theorem~\ref{r>1} (under the natural assumption that
$\Sc_{\beta^i} V$ is a subrepresentation of $A \otimes \Sc_\alpha V$
for $i=1, \ldots, r$). In general, the resolution we give may not be
minimal, see Example~\ref{nonminimalexample}. However, we give an
explicit closed form description for the minimal free resolution in
Corollary~\ref{singlecolumns} for the special case when $\beta^i$ and
$\alpha$ differ in only one entry for each $i = 1, \dots, r$. We call
the minimal resolutions of maps of the form \eqref{pierimap2} {\bf
  Pieri resolutions}.

The map $\phi(\alpha, \beta)$ (and hence $\phi(\alpha; \beta^1, \dots,
\beta^r)$) can be calculated (up to a scalar multiple) in Macaulay 2
using the {\tt PieriMaps} package (see \cite{pierimaps}).

\begin{eg} 
  Let $n=3$, $\alpha = (3,1,0)$, $r=2$, and $\beta^1 = (5,1,0)$ and
  $\beta^2 = (3,2,0)$. Representing the module $A \otimes \Sc_\lambda
  V$ by the Young diagram of $\lambda$ (our convention for partitions:
  the diagram of $\lambda$ has $\lambda_i$ boxes in the $i$th column),
  we get the following resolution:
  \[
  0 \to \tiny \tableau[scY]{,,|,,||||} \to \tableau[scY]{,|,||||}
  \oplus \tableau[scY]{,,|,,||} \to \tableau[scY]{,|,||} \oplus
  \tableau[scY]{,|||||} \to \tableau[scY]{,|||} \to M \to 0
  \]
  where $M = \Sc_{(3,1,0)} V \oplus \Sc_{(4,1,0)} V \oplus
  \Sc_{(3,1,1)} V \oplus \Sc_{(4,1,1)} V$.
\end{eg}

We remark that the techniques of \cite{efw} are limited to
characteristic 0 because of the failure of the Borel--Weil--Bott
theorem in positive characteristic, and our techniques are limited to
characteristic 0 because semisimplicity of the general linear group
fails otherwise, which means that nonzero equivariant maps of the form
\eqref{pierimap} often do not exist.

Also included in Section~\ref{pieriresolutionsection} is how
equivariant resolutions can be constructed when $A$ is replaced by the
exterior algebra $B = \bigwedge V$. The resolutions one obtains are
infinite in length, but still simple to describe combinatorially.

In Section~\ref{otherclassicalgroups} we generalize the results of
\cite{efw} to other classical groups. When $G$ is an orthogonal or
symplectic group, we have a standard representation $F$ (a vector
space with a symmetric or skew-symmetric nondegenerate bilinear
form). The highest weights of irreducible representations of the group
$G$ occurring in the tensor powers on $F$ still correspond to
partitions. The constructions in the case of the general linear group
are functorial and hence extend to vector bundles. We construct the
analogues of a family of graded equivariant $\Sym(F)$-modules $M$
which is analogous to the cokernel of Pieri maps in the $\GL(V)$ case
by considering Pieri resolutions of homogeneous bundles over certain
homogeneous spaces for $G$. Then we use the geometric technique (see
Theorem~\ref{geometrictechnique}) and the Borel--Weil--Bott theorem
for the group $G$ to describe the minimal free resolution of the
module $M$.  They are not pure but can still be considered to be the
analogues of the complexes from \cite{efw}. The Lie types
$\mathrm{B}_n$, $\mathrm{C}_n$, and $\mathrm{D}_n$ are treated in
separate subsections. The arguments here are more delicate, since the
resolutions are constructed as iterated mapping cones, and for
example, in the calculations for type B, one must analyze a connecting
homomorphism in a long exact sequence to prove that some repeating
representations cancel.

Section~\ref{equivariantbssection} is concerned with a possible
equivariant analogue of the Boij--S\"oderberg conjectures which were
proved in \cite{es}. If $M$ is an equivariant module, let ${\bf
  F}_\bullet$ be its equivariant minimal free resolution. We define
its equivariant Betti table $\B(M)$ as follows: if ${\bf F}_i =
\bigoplus_j A(-j) \otimes V_{i,j}$, then $\B(M)_{i,j}$ is the
character of $V_{i,j}$. 

The strong version of the conjecture says that given any finite length
equivariant module $M$ with equivariant resolution ${\bf F}_\bullet$,
there exist representations $W, W_1, \dots, W_r$ such that $W \otimes
{\bf F}_\bullet$ has a filtration with associated graded
\begin{align} \label{complexconjecture} \operatorname{gr}(W \otimes
  {\bf F}_\bullet) \cong \bigoplus_{i=1}^r W_i \otimes {\bf
    F}(\alpha^i, \beta^i)_\bullet,
\end{align}
where ${\bf F}(\alpha^i, \beta^i)_\bullet$ is the minimal free
resolution of the map $\phi(\alpha^i, \beta^i)$. 

The weak version of the conjecture replaces the isomorphism of
complexes in \eqref{complexconjecture} with an equality of equivariant
Betti tables. If we remove the adjective ``equivariant,'' then the
weak version of the conjecture is a theorem of Eisenbud and Schreyer
\cite[Theorem 0.2]{es}. Furthermore, their result holds over any
field, and the finite length condition can be replaced by an arbitrary
codimension. The weak form of the conjecture would already be
interesting from the point of view of cohomology tables of homogeneous
bundles on projective space. In \cite{es}, a bilinear pairing is
defined between minimal resolutions over $\Sym(V)$ and vector bundles
on the projective space $\P(V)$ which reveals a duality between the
two. This bilinear pairing can also be defined in an equivariant way,
and one hopes that a similar kind of duality holds in an equivariant
sense. 

We present some examples of decompositions predicted by the weak
version of the conjecture in Section~\ref{equivariantbssection}. We
also provide some partial results in this direction (see
Proposition~\ref{simplicialcase}) and discuss some of the difficulties
in trying to extend the proof of Eisenbud and Schreyer to the
equivariant setting and in trying to find counterexamples to the
existence of such decompositions.

\subsection*{New in this version.}

In previous versions of this article (including the published version), we mistakingly passed over some of Olver's results in \S\ref{olvermapsection}. We added a footnote to the introduction, added some remarks to the beginning of \S\ref{purefreesection}, and changed some text in \S\ref{olvermapsection} to correct this mistake.

\subsection*{Acknowledgements.}
The authors would like to thank David Eisenbud and an anonymous
referee for reading previous drafts of this article and for making
numerous helpful suggestions. We also thank David Handelman for
pointing out the reference \cite{handelman}. Steven Sam was supported
by an Akamai Presidential Fellowship while this work was done. Jerzy
Weyman was partially supported by NSF grant 0901185.

\section{Background.}

In Section~\ref{reptheorysection}, we define our notation for
partitions and representations of $\GL(V)$ (which is slightly
nonstandard). In Section~\ref{olvermapsection} we give Olver's
explicit description of the inclusion arising from a Pieri-type tensor
product decomposition.

\subsection{Partitions and representation
  theory.} \label{reptheorysection}

Let $\alpha$ denote a partition, i.e., a sequence $\alpha = (\alpha_1,
\ldots, \alpha_n)$ with $\alpha_i \in \Z$ and $\alpha_1 \ge \alpha_2
\ge \cdots \ge \alpha_n \ge 0$. We let $\ell(\alpha)$ denote the {\bf
  length} of $\alpha$, which is defined to be the largest $m$ such
that $\alpha_m \ne 0$. We represent $\alpha$ by its Young diagram
$D(\alpha)$ with $\alpha_i$ boxes in the $i$th column.\footnote{We
  remark here that the usual (English) convention of drawing
  partitions is to have $\alpha_i$ boxes in the $i$th row, but this
  transposed way gives a compact notation for writing down
  resolutions.} The {\bf dual partition} $\alpha^*$ is defined by
setting $\alpha_i^*$ to be the number of $j$ such that $\alpha_j \ge
i$, or equivalently, the number of boxes in the $i$th row of
$D(\alpha)$. The notation $\alpha \subseteq \beta$ means that
$\alpha_i \le \beta_i$ for all $i$, or equivalently, $D(\alpha)
\subseteq D(\beta)$, and in this case, $\beta / \alpha = D(\beta /
\alpha)$ refers to the skew diagram $D(\beta) \setminus D(\alpha)$.

For brevity, we will say $(\beta, \alpha) \in \VS$ to mean that $\beta
\supseteq \alpha$ and $\beta / \alpha$ is a {\bf vertical strip},
i.e., $\beta_i \le \alpha_{i-1}$ for all $i$, or equivalently, that
there is at most one box in each row of $\beta / \alpha$. Analogously,
$(\beta, \alpha) \in \HS$ will mean that $\beta / \alpha$ is a {\bf
  horizontal strip}, i.e., $\beta^* / \alpha^*$ is a vertical
strip. The notation $(\beta, \alpha) \notin \VS$ shall mean that
either $\alpha \not\subseteq \beta$, or that $\alpha \subseteq \beta$
but $\beta/\alpha$ is not a vertical strip, and similarly for $(\beta,
\alpha) \notin \HS$.

The {\bf union} of two partitions $\beta \cup \beta'$ is defined to
have $i$th part $\max(\beta_i, \beta'_i)$, so that $D(\beta \cup
\beta') = D(\beta) \cup D(\beta')$. We will also use the notation
$\alpha < \beta$ ({\bf lexicographic ordering}) to mean that the first
nonzero entry of $(\beta_1 - \alpha_1, \beta_2 - \alpha_2, \dots )$ is
positive. Note that $<$ is a total ordering which extends $\subseteq$.

Fix a vector space $V$ with an ordered basis $x_1, \dots, x_n$. This
ordered basis determines a maximal torus and Borel subgroup $T \subset
B \subset \GL(V)$. We identify partitions $\alpha$ with dominant
weights of $\GL(V)$, and let $\Sc_\alpha V$ denote the irreducible
representation of $\GL(V)$ with highest weight $\alpha$, thought of as
a factor module of $\Sym^\alpha(V) = \Sym^{\alpha_1}(V) \otimes \cdots
\otimes \Sym^{\alpha_n}(V)$. We identify the elements of $\Sc_\alpha
V$ with linear combinations of fillings of $D(\alpha)$ using elements
of $\{1, \dots, n\}$, where $i$ is identified with the basis element
$x_i$, which are weakly increasing top to bottom along columns modulo
certain relations (see \cite{pierimaps} for more details). A basis is
given by semistandard Young tableaux, i.e., those fillings which are
strictly increasing from left to right along rows. Under this
identification, a highest weight vector is given by the tableau with
all boxes in column $i$ labelled with an $i$. We refer to this tableau
as the {\bf canonical tableau}. By the Weyl character formula, see
\cite[I, Appendix A.8]{macdonald}, the character of $\Sc_\alpha V$ is
given by the {\bf Schur polynomial}
\begin{align} \label{weylcharacterformula} s_\alpha = s_\alpha(x_1,
  \dots, x_n) = \frac{\det(x_j^{\alpha_i + n -
      i})_{i,j=1}^n}{\det(x_j^{n-i})_{i,j=1}^n}.
\end{align}
In Section~\ref{otherclassicalgroups}, we will need to know that the
definition of $\Sc_\alpha V$ is functorial with respect to $V$ and
extends to vector bundles.

Let $R(\GL(V))$ be the Grothendieck ring of the tensor category of
finite-dimensional rational representations of $\GL(V)$. By
semisimplicity, every representation can be written uniquely as a
direct sum of irreducible representations $\Sc_\alpha V \otimes
(\bigwedge V)^{\otimes r}$ where $\alpha = (\alpha_1 \ge \alpha_2 \ge
\cdots \ge \alpha_n = 0)$ and $r \in \Z$. Let $\Z[x_1, \dots,
x_n]^{\SS_n}$ denote the ring of symmetric functions, and define the
{\bf character} $\ch \colon R(\GL(V)) \to \Z[x_1, \dots,
x_n]^{\SS_n}[(x_1\cdots x_n)^{-1}]$, which is a ring isomorphism given
by
\begin{align} \label{character} \ch \Big( \bigoplus_\lambda
  \big(\Sc_\lambda V \otimes ( \bigwedge V )^{\otimes {r_\lambda}}
  \big)^{\oplus c_{\lambda}} \Big) = \sum_\lambda c_\lambda s_\lambda
  \cdot (x_1 \cdots x_n)^{r_\lambda}.
\end{align}

\subsection{Olver's description of Pieri
  inclusions.} \label{olvermapsection}

The following formula is crucial for the existence of our equivariant
resolutions. 

\begin{theorem}[Pieri's formula] Let $\alpha$ be a partition, and
  let $b$ be a positive integer. We have isomorphisms of
  $\GL(V)$-modules
  \begin{align*}
    \Sc_b V \otimes \Sc_\alpha V &\cong \bigoplus_{\substack{(\beta,
        \alpha) \in \VS \\ |\beta/\alpha| = b}} \Sc_\beta V,\\
    \bigwedge^b V \otimes \Sc_\alpha V &\cong
    \bigoplus_{\substack{(\beta, \alpha) \in \HS \\ |\beta/\alpha| =
        b}} \Sc_\beta V.
  \end{align*}
\end{theorem}

\begin{proof} See \cite[(5.16), (5.17)]{macdonald} or \cite[Corollary
  2.3.5]{weyman}. Note that in both sources, the convention for Young
  diagrams is transpose to ours, and that in \cite{weyman}, $L_\lambda
  E$ is an irreducible representation with highest weight $\lambda^*$.
\end{proof}

In particular, we get inclusions $\Sc_\beta V \to \Sc_b V \otimes
\Sc_\alpha V$, which are well-defined up to a (nonzero) scalar
multiple. We call such maps {\bf Pieri inclusions}.

In fact, one can describe this map explicitly with respect to the
basis of semistandard Young tableaux. The following description for
the case when $b=1$, i.e., $\beta / \alpha$ is a single box, and we
have a map $\Sc_\beta V \to V \otimes \Sc_\alpha V$, comes from
\cite[\S 6]{olver} where they are called polarization maps. The
general case is developed from this case in \cite[\S 8]{olver} by a careful choice of coefficients. We offer a similar, but alternative derivation from the $b=1$ case which does not involve as much attention to the coefficients.

First, we work with more general ``shapes.'' That is, diagrams
$D(\lambda)$ obtained by dropping the requirement that $\lambda_1 \ge
\lambda_2 \ge \cdots$. Given a tableau $T$ with underlying shape
$\lambda$ and indices $i<j$, set $\tau_{ij}(T)$ to be the sum of all
fillings of shapes obtained by removing a box along with its label
from the $j$th column (and the boxes below it are shifted up to fill
in the hole) and appending it to the end of the $i$th column. There
are of course $\lambda_j$ such ways to do so counting multiplicity. If
$i=0$, then we consider the box to be in the ``0th column'' (this will
correspond to the $V$ in $V \otimes \Sc_\alpha V$). Given an
increasing sequence $J = (j_1 < j_2 < \cdots < j_r)$, we define
$\tau_J = \tau_{j_{r-1}j_r} \circ \cdots \circ \tau_{j_1j_2}$, and
define $\# J = r$. The fillings obtained need not be
semistandard, but they are well-defined elements of $V \otimes
\Sc_\alpha V$ (see Section~\ref{reptheorysection}).

Now suppose that $\beta / \alpha$ is a single box in the $k$th
column. Given our basis $\{x_1, \dots, x_n\}$ of $V$, the basis
elements of $\Sc_\beta V$ are identified with semistandard tableaux of
shape $\beta$ with labels $\{1, \dots, n\}$, and the basis elements of
$V \otimes \Sc_\alpha V$ are identified with elements $x_i \otimes T$
where $1 \le i \le n$ and $T$ is a semistandard tableau of shape
$\alpha$ (the variable can be thought of as the ``0th column.'') Let
$B_k$ be the set of strictly increasing sequences $j_1 < j_2 < \cdots
< j_r$ (of all lengths $r$) such that $j_1 = 0$ and $j_r = k$. For $J
\in B_k$, define the coefficients
\begin{align} \label{cJ} 
  c_J = \prod_{i=2}^{\#J - 1} (\beta_{j_i} - \beta_k + k - j_i)
\end{align}
(the empty product is 1). Then the Pieri inclusion is 
\begin{align} \label{olvermap} \sum_{J \in B_k} \frac{(-1)^{\# J}
    \tau_J}{c_J}.
\end{align}
Essentially, the Pieri inclusion is obtained by summing ``all possible
ways to remove a box from a semistandard tableau of shape $\beta$ to
get a variable in $V$, times the remaining filling of $\alpha$.'' Of
course, in general, this filling will not be semistandard but we can
use the relations in $\Sc_\alpha V$ to write them in terms of a
semistandard basis. Details and some examples can be found in
\cite{pierimaps}. 

In order to get the general case, one first picks a filtration of
partitions $\beta = \alpha^0 \supset \alpha^1 \supset \cdots \supset
\alpha^b = \alpha$ where $b = |\beta / \alpha|$ and each $\alpha^j /
\alpha^{j+1}$ is a single box. Composing the Olver maps, one gets a
map 
\[
\Sc_\beta V \to V \otimes \Sc_{\alpha^1} V \to \cdots \to V^{\otimes
  b} \otimes \Sc_\alpha V.
\]
To get the desired inclusion, we compose this with $V^{\otimes b}
\otimes \Sc_\alpha V \to \Sc_b V \otimes \Sc_\alpha V$ where the map
on the first component is the canonical projection of a tensor power
onto a symmetric power, and the second component is the identity map.

The following lemma explains one way to extend the above definition of Olver maps to the case that $|\beta/\alpha|>1$.

\begin{lemma}[Olver] \label{olvercompositionlemma} With the above notation,
  the composition $\Sc_\beta V \to \Sc_b V \otimes \Sc_\alpha V$ is
  nonzero. In fact, up to nonzero scalar multiples, the following
  diagram
  \[
  \xymatrix{ \Sc_\beta V \ar[d]_-\psi & V \otimes V \otimes
    \Sc_{\alpha^2} V \ar[r]^-{p \otimes \psi} \ar[d]_-{p \otimes 1} &
    \cdots \ar[r]^-{p \otimes \psi} & \Sc_{b-2} V \otimes V \otimes
    \Sc_{\alpha^{b-1}} V \ar[r]^-{p \otimes \psi} \ar[d]_-{p \otimes
      1} & \Sc_{b-1} V \otimes V \otimes \Sc_\alpha V \ar[d]_-{p
      \otimes 1} \\ V \otimes \Sc_{\alpha^1} V \ar[ur]^-{1 \otimes
      \psi} & \Sc_2 V \otimes \Sc_{\alpha^2} V \ar[ur]^-{1 \otimes
      \psi} & \cdots & \Sc_{b-1} V \otimes \Sc_{\alpha^{b-1}} V
    \ar[ur]^-{1 \otimes \psi} & \Sc_b V \otimes \Sc_\alpha V }
  \]
  commutes, where $p \colon \Sc_i V \otimes V \to \Sc_{i+1} V$ is the
  usual projection map, and $\psi$ denotes the Olver map as described
  above. Furthermore, the map $\Sc_\beta V \to \Sc_b V \otimes
  \Sc_\alpha V$ is nonzero for all filtrations $\beta = \alpha^0
  \supset \alpha^1 \supset \cdots \supset \alpha^b = \alpha$.
\end{lemma}

\begin{proof} Let $V$ have an ordered basis $x_1, \dots, x_n$ so that
  we may identify symmetric powers of $V$ with monomials in the
  $x_i$. Suppose that $\alpha$ is obtained from $\beta$ by removing
  boxes in columns $c_1 \ge c_2 \ge \cdots \ge c_b$.

  We first show that if the filtration is picked such that
  $\alpha^{i-1} / \alpha^i$ is a single box in column $c_i$, then the
  composition is nonzero. Let $T_\lambda$ be the canonical tableau in
  $\Sc_\lambda V$. The image of $T_\beta$ under the composition
  $\Sc_\beta V \to \Sc_b V \otimes \Sc_\alpha V$ is
  \[
  \beta_{c_1} \beta_{c_2} \cdots \beta_{c_b} x_{c_1} x_{c_2} \cdots
  x_{c_b} \otimes T_\alpha + o(T_\alpha)
  \]
  where $o(T_\alpha)$ is a sum of tensors whose $\Sc_\alpha V$
  component is a vector with weight lower than $T_\alpha$. This is
  clear by induction on $b$ from Olver's description: of all the
  increasing sequences $J$ in \eqref{olvermap} involved in the map
  $\Sc_{i-1} V \otimes \Sc_{\alpha^{i-1}} \to \Sc_{i-1} V \otimes V
  \otimes \Sc_{\alpha^i}$, only $J = (0 < c_i)$ has the property that
  $x_{c_1} \cdots x_{c_{i-1}} \otimes \tau_J(T_{\alpha^{i-1}})$ is
  written as a linear combination of basis vectors which can contain
  $x_{c_1}x_{c_2} \cdots x_{c_i} \otimes T_\alpha$ with a nonzero
  coefficient when mapped to $\Sc_i V \otimes \Sc_\alpha V$.

  Now we show that the map is nonzero independent of the chosen
  filtration. We first assume that $b=2$. Let $i=c_1$ and $j=c_2$. We
  may assume that $i<j$ or there is nothing to show. Then of course
  $\beta_i > \beta_j$ because $\beta / \alpha$ is a vertical strip.
  Let $\beta'$ be $\beta$ with a box in column $j$ removed. Let
  $\phi_1$ be the first map $\Sc_\beta V \to V \otimes \Sc_{\beta'}
  V$, and let $\phi_2$ be the second map $V \otimes S_{\beta'} V \to V
  \otimes V \otimes \Sc_\alpha V$, where the map on the first factor
  of $V$ under $\phi_2$ is the identity.

  There are two basis elements of $V \otimes \Sc_{\beta'} V$ with
  nonzero coefficient in $\phi_1(T_\beta)$ which can map to $x_ix_j
  \otimes T_\alpha$. Namely, the first is $C_1 x_j \otimes
  T_{\beta'}$, and the second is $C_2 x_i \otimes L$ for some
  coefficients $C_1, C_2$, where $L$ is the tableau with $\beta_k$
  $k$'s in the $k$th column if $k\ne i$ and $k \ne j$, contains
  $\beta_i-1$ $i$'s and 1 $j$ in the $i$th column, and $\beta_j-1$
  $j$'s in the $j$th column. Using Olver's description, the
  coefficients are
  \[
  C_1 = \beta_j, \quad C_2 = - \frac{\beta_i \beta_j}{\beta_i -
    \beta_j + j - i}.
  \]
  Now, we also have
  \[
  \phi_2(T_{\beta'}) = \beta_i x_i T_\alpha + o(T_\alpha), \quad
  \phi_2(L) = x_j T_\alpha + o(T_\alpha),
  \]
  so putting it all together,
  \begin{equation} \label{nonzeroolver}
    \begin{split}
      \phi_2(\phi_1(T_\beta)) &= C_1 x_j \phi_2(T_{\beta'}) + C_2 x_i
      \phi_2(L) + o(T_\alpha) \\
      &= \beta_i \beta_j x_j x_i T_\alpha - \frac{\beta_i
        \beta_j}{\beta_i - \beta_j + j - i} x_i x_j T_\alpha + o(T_\alpha) \\
      &= \beta_i \beta_j x_i x_j \left(1 - \frac{1}{\beta_i - \beta_j
          + j - i}\right) T_\alpha + o(T_\alpha).
    \end{split}
  \end{equation}
  Finally, we know that $\beta_i - \beta_j + j - i \ge 2$, so the
  coefficient of $T_\alpha$ in the last expression is nonzero. Hence
  for $r=2$, both ways of composing Olver maps give nonzero maps, and
  hence must be scalar multiples of each other.

  For the general case $r \ge 2$, note that any permutation of the
  order of box removals is valid because $\beta / \alpha$ is a
  vertical strip. Hence, any two permutations of compositions are
  scalar multiples of each other because the symmetric group is
  generated by transpositions.
\end{proof}

We take a minute to discuss the scalar multiples that appear in the
above proof. Let $f_1$ be the composition $\Sc_\beta V \to \Sc_b V
\otimes \Sc_\alpha V$ obtained by removing the boxes in increasing
order of column index $c_1 \le c_2 \le \cdots \le c_b$, and for a
permutation $\sigma \in \SS_b$, let $f_\sigma$ be the composition
obtained by removing the boxes in the order $c_{\sigma^{-1}(1)},
c_{\sigma^{-1}(2)}, \dots, c_{\sigma^{-1}(b)}$. If $\beta / \alpha$
has $b_i$ boxes in the $i$th column, then $f_\sigma = f_\tau$ if
$\sigma$ and $\tau$ represent the same left coset in $\SS_b /
(\SS_{b_1} \times \cdots \times \SS_{b_n})$, where $\SS_{b_1} \times
\cdots \times \SS_{b_n}$ is the subgroup of permutations which maps
$\{b_1 + \cdots + b_{i-1} + 1,\dots,b_1 + \cdots + b_i\}$ amongst
themselves for $i=1,\dots,n$. We have seen that if $s_i$ is the
transposition $(i,i+1)$ and $c_i \ne c_{i+1}$, then
\[
f_{s_i} = \left(1 - \frac{1}{\beta_{c_i} 
 - \beta_{c_{i+1}} + c_{i+1} - c_i} \right) f_1.
\]
The formula for $f_\sigma$ in terms of $f_1$ is complicated in
general, so we content ourselves with this special case. It is possible to make a careful choice of normalization of the Olver maps so that we get commutative squares, see \cite[\S 8]{olver} for details.

The above proof also shows that we can replace symmetric powers with
exterior powers.

\begin{corollary} \label{exteriorolver} Using the notation of this
  section, replacing the map $V^{\otimes b} \otimes \Sc_\alpha V \to
  \Sc_b V \otimes \Sc_\alpha V$ by $V^{\otimes b} \otimes \Sc_\alpha V
  \to \bigwedge^b V \otimes \Sc_\alpha V$ gives a nonzero composition
  $\Sc_\beta V \to \bigwedge^b V \otimes \Sc_\alpha V$ whenever $\beta
  / \alpha$ is a horizontal strip.
\end{corollary}

\begin{proof} It is enough to note that in modifying the proof of
  Lemma~\ref{olvercompositionlemma} to work for $\bigwedge V$, the
  only change is in \eqref{nonzeroolver}, where $\DS 1 -
  \frac{1}{\beta_i - \beta_j + j - i}$ is replaced by $\DS 1 +
  \frac{1}{\beta_i - \beta_j + j - i}$, which is also nonzero.
\end{proof}

\section{Equivariant resolutions for the general linear
  group.} \label{equivariantressection}

In Section~\ref{purefreesection}, we recall the construction of
Eisenbud, Fl\o ystad, and Weyman for pure free equivariant resolutions
(Theorem~\ref{equivariantres}) which resolves a special class of Pieri
inclusions. We generalize this construction to the case of an
arbitrary Pieri inclusion $\phi(\alpha, \beta)$ and to direct sums of
Pieri inclusions. In order to describe our resolution we introduce
some combinatorial notions in Section~\ref{criticalboxessection}. The
actual resolution (and some generalizations) are given in
Section~\ref{pieriresolutionsection}, and examples are given in
Section~\ref{examplesection}.

\subsection{Pure free resolutions.} \label{purefreesection}

In this section, we describe the equivariant pure resolutions of
\cite{efw} and prove their acyclicity using Lemma~\ref{olvercompositionlemma}. We note that Lemma~\ref{acycliclemma} and Theorem~\ref{equivariantres} can also be found in \cite[\S 8]{olver} and \cite[Theorem 8.11]{olver}, respectively.

First we recall the notion of a pure free resolution. Every graded
$A$-module $M$ has a minimal graded free resolution $F_\bullet \to M
\to 0$ which is unique up to isomorphism. The number of minimal
generators of degree $i$ of $F_j$ is $\B(M)_{i,j}$, which are the {\bf
  graded Betti numbers} of $M$. The usual convention for representing
Betti numbers is via a {\bf Betti diagram/table}: this is an array of
numbers whose $i$th column and $j$th row contains
$\B(M)_{i,j-i}$. Thinking of $K = A/(x_1, \dots, x_n)$ as a trivial
$A$-module, we note that $\B(M)_{i,j} = \dim_K \Tor^A_j(M, K)_i$. Then
$M$ has a {\bf pure free resolution} if for each $i$, $\B(M)_{i,j}$ is
nonzero for at most one value of $j$, i.e., each syzygy module of a
minimal free resolution of $M$ is generated in a single degree. We
also say that $\B(M)_{i,j}$ is a {\bf pure Betti diagram/table}. In
this case, we define $d_i$ to be the degree of $F_i$, and the sequence
$d = (d_i)$ is the {\bf degree} of $F_\bullet$.

Let $\alpha$ and $\beta$ be partitions such that $\beta / \alpha$ is a
vertical strip of size $b > 0$. By choosing a scalar multiple for the
Pieri inclusion $\Sc_\beta V \to \Sc_b V \otimes \Sc_\alpha V$, we get
a uniquely determined (up to a nonzero scalar) equivariant map of
$A$-modules
\[
\phi(\alpha ,\beta ) \colon A(-b) \otimes \Sc_\beta V \to A \otimes
\Sc_\alpha V
\]
of degree 0. (Here $A(a)$ denotes a grading shift by $a$.) Our goal is
to describe an equivariant minimal free resolution of the $A$-modules
$\coker(\phi(\alpha, \beta))$. First, we recall the case when the
cokernel has finite length. This corresponds to the case when $\beta /
\alpha$ contains boxes only in the first column.

Set $e_1 = \beta_1 - \alpha_1$; for $i > 1$ set $e_i = \alpha_{i-1} -
\alpha_i + 1$. Define a sequence $d = (d_0, \dots, d_n)$ by $d_0 = 0$
and $d_i = e_1 + \cdots + e_i$ for $i \ge 1$, and define some
partitions
\[
\alpha(d, i) = (\alpha_1 + e_1, \alpha_2 + e_2, \dots, \alpha_i + e_i,
\alpha_{i+1}, \dots, \alpha_n)
\]
for $1 \le i \le n$. Define graded $A$-modules ${\bf F}(d)_i$ for $0
\le i \le n$ by
\begin{align*}
  {\bf F}(d)_0 &= A \otimes \Sc_\alpha V\\
  {\bf F}(d)_i &= A(-e_1 - \cdots - e_i) \otimes \Sc_{\alpha(d,i)} V,
  \quad (1 \le i \le n).
\end{align*}
The natural action of $\GL(V)$ on $A = \bigoplus_{i \ge 0} \Sc_{i}
V$ and of $\GL(V)$ on $\Sc_{\alpha(d,i)} V$ gives an action of
$\GL(V)$ on ${\bf F}(d)_i$. Picking a Pieri inclusion
\[ 
\psi \colon \Sc_{\alpha(d,i)} V \to \Sc_{e_i} V \otimes
\Sc_{\alpha(d,i-1)} V
\]
and identifying $\Sc_{e_i} V = \Sym^{e_i} V$ gives a degree 0 map
$\partial_i \colon {\bf F}(d)_i \to {\bf F}(d)_{i-1}$ defined by
$\partial_i(p(x) \otimes v) = p(x) \cdot \psi(v)$ where $p(x) \in A$
and $v \in \Sc_{\alpha(d,i)} V$.

Olver's description of the Pieri inclusion gives the following
lemma, which greatly simplifies the proof of
Theorem~\ref{equivariantres} compared to the original proof found in
\cite{efw}. 

\begin{lemma}[Olver] \label{acycliclemma} Pick $\mu / \nu \in \VS$. Given a
  partition $\lambda$ such that $\lambda / \mu \in \VS$ and $\lambda /
  \nu \in \VS$, the composition $\Sc_\lambda V \to A \otimes \Sc_\mu V
  \to A \otimes \Sc_\nu V$, where the first map is a Pieri inclusion
  and the second map is induced by a Pieri inclusion, is nonzero.
\end{lemma}

\begin{proof} Pick a filtration of partitions $\lambda = \lambda^0
  \supset \lambda^1 \supset \cdots \supset \lambda^s = \nu$ such that
  $\lambda^r = \mu$ for $r = |\lambda / \mu|$, and such that
  $\lambda^j / \lambda^{j+1}$ is a single box for all $j$. By
  Lemma~\ref{olvercompositionlemma}, $\partial_i \colon A \otimes
  \Sc_\mu V \to A \otimes \Sc_\nu V$ is equal to a composition of
  nonzero scalar multiples of Olver maps, one for each piece of the
  filtration $\lambda^r \supset \lambda^{r+1} \supset \cdots \supset
  \lambda^s$. Again by Lemma~\ref{olvercompositionlemma}, the
  composition of $\partial_i$ with the composition of Olver maps
  $\Sc_\lambda V \to A \otimes \Sc_\mu V$ corresponding to the
  filtration $\lambda^0 \supset \lambda^1 \supset \cdots \supset
  \lambda^r$ is nonzero, which proves the claim.
\end{proof}

For the following theorem, we point out that the notation $(\lambda,
\beta) \notin \VS$ means that either $\beta \not\subseteq \lambda$, or
that $\beta \subseteq \lambda$, but $\lambda / \beta$ is not a
vertical strip.

\begin{theorem}[Olver, Eisenbud--Fl{\o}ystad--Weyman] \label{equivariantres}
  With the notation above,
  \[
  \xymatrix{ 0 \ar[r] & {\bf F}(d)_n \ar[r]^-{\partial_n} & \cdots
    \ar[r]^-{\partial_2} & {\bf F}(d)_1 \ar[r]^-{\partial_1} & {\bf
      F}(d)_0 }
  \]
  is a $\GL(V)$-equivariant minimal graded free resolution of $M(d) =
  \coker \partial_1 = \coker \phi(\alpha, \beta)$, which is pure of
  degree $d$. Furthermore, there is an isomorphism (as
  $\GL(V)$-representations) 
  \[
  M(d) \cong \bigoplus_{\substack{ (\lambda, \alpha) \in \VS \\
      (\lambda, \beta) \notin \VS}} \Sc_\lambda V.
  \]
\end{theorem}

The proof in \cite{efw} uses the Borel--Weil--Bott theorem. However,
we appeal only to Olver's description of the Pieri inclusion.

\begin{proof}[Proof of Theorem~\ref{equivariantres}] 
  That ${\bf F}(d)_\bullet$ is a complex is obvious: we have picked
  the partitions $\alpha(d,i)$ so that for any partition $\lambda$
  such that $(\lambda, \alpha(d,i+1)) \in \VS$, we have $(\lambda,
  \alpha(d,i-1)) \notin \VS$, and then we use that the $\Sc_\lambda V$
  are irreducible representations of $\GL(V)$. That it is acyclic
  follows from almost the same reason: if $(\lambda, \alpha(d,i-1))
  \notin \VS$ and $(\lambda, \alpha(d,i)) \in \VS$, then we have
  $(\lambda, \alpha(d,i+1)) \in \VS$ by our choices of
  $\alpha(d,i)$. However, one needs to know that the image of
  $\Sc_\lambda V$ under the map ${\bf F}(d)_{i+1} \to {\bf F}(d)_i$ is
  not zero, and this is the content of Lemma~\ref{acycliclemma}.
\end{proof}

\begin{eg} Let $\alpha = (3,1,0,0)$ and $\beta = (5,1,0,0)$, so that
  $d = (0,2,5,7,8)$ and $e = (2,3,2,1)$. Then $\alpha(d,i)$ is the
  partition such that $D(\alpha(d,i))$ is the subdiagram of the
  following diagram consisting of boxes with labels $\le i$:
  \[
  \tableau[scY]{0,0,3,4 | 0,2,3 | 0,2 | 1,2 | 1},
  \]
  and the complex ${\bf F}(d)_\bullet$ looks like (where we use
  $\lambda$ as shorthand for $\Sc_\lambda V$)
  \[
  0 \to (5,4,2,1) \to (5,4,2,0) \to (5,4,0,0) \to (5,1,0,0) \to
  (3,1,0,0).
  \]
\end{eg}

\begin{rmk}[Pure free resolutions over $\bigwedge
  V$] \label{exteriorpurefree} \rm Let $B = \bigwedge V$ be the
  exterior algebra of $V$. Given $\alpha$ and $\beta$ as before such
  that $\beta / \alpha$ contains boxes only in the first
  column. Define $\alpha(d,1) = \beta$, and $\alpha(d,i)$ for $i > 1$
  by
  \[
  \alpha(d,i)_j = \begin{cases} \beta_1, & \text{if } j=1,\\
    \alpha_{j-1} + 1, & \text{if } 2 \le j \le \min(i,n+1),\\
    \alpha_j, & \text{if } i < j \le n+1,\\
    1, & \text{if } n+1 < j \le i.
  \end{cases}
  \]
  Also, set ${\bf F}'_0 = B \otimes \Sc_{\alpha^*} V$ and ${\bf F}'_i
  = B(-|\alpha(d,i) / \alpha|) \otimes \Sc_{\alpha(d,i)^*} V$ for
  $i>1$. Picking Pieri inclusions
  \[
  \Sc_{\alpha(d,i)^*} V \to \bigwedge^{e_i} V \otimes
  \Sc_{\alpha(d,i-1)^*} V
  \]
  where $e_i = |\alpha(d,i) / \alpha(d,i-1)|$ gives $B$-linear
  equivariant differentials ${\bf F}'_i \to {\bf F}'_{i-1}$. An
  analogue of Lemma~\ref{acycliclemma} using
  Corollary~\ref{exteriorolver} shows that ${\bf F}'_\bullet$ is a
  free resolution of the cokernel of ${\bf F}'_1 \to {\bf F}'_0$.
\end{rmk}

\subsection{Critical boxes and admissible
  subsets.}\label{criticalboxessection} 

We consider partitions with at most $n$ parts and identify their Young
diagrams as subsets of a grid $L$ of boxes with $n$ columns, going
infinitely downwards. The boxes in this grid can be thought of as
pairs $(j,i)$ with $j=0,1,\dots$ and $1\le i\le n$. In this case, $i$
is the {\bf column index} of $(j,i)$. Recall that the notation
$\lambda > \mu$ refers to lexicographic ordering of partitions (see
Section~\ref{reptheorysection}).

Let $\alpha$ and $\beta$ be partitions such that $(\beta, \alpha) \in
\VS$. Let $c_1 < c_2 < \cdots < c_m$ denote the indices of the columns
of the skew shape $\beta / \alpha$. Define the set of {\bf critical
  boxes} as follows
\[
C(\alpha, \beta) = \{ (\alpha_{j-1}+1, j) \in L \mid c_1 < j \le n \}
\]
We shall sometimes refer to the critical boxes by their column
indices. Given a subset $J \subseteq C(\alpha, \beta)$, we denote by
$\beta(J)$ the smallest partition whose Young diagram contains both
$\beta$ and $J$. The subsets $J \subseteq C(\alpha, \beta)$ whose
column indices are unions of subsets of consecutive integers of the
form $\{c_i+1, c_i+2, \dots,j\}$ are {\bf admissible}. By convention,
the empty set is admissible. The set of all admissible subsets is
denoted $\Ad(\alpha; \beta)$, and we define
\[
\Ad(\alpha; \beta)_i = \{ J \in \Ad(\alpha; \beta) \mid \#J+1 = i\},
\]
where $\#\emptyset = 0$. This definition will only be used in
Theorem~\ref{r=1}.

\begin{eg} Let $n=8$, $\alpha = (4,4,3,2,1,0,0,0)$, and $\beta =
  (5,4,3,2,2,1,0,0)$. In the following picture, $\alpha$ is the set of
  white boxes, $\beta / \alpha$ is the set of framed boxes, and the
  critical boxes are marked with the symbol $\times$. The black boxes
  are holes and are not part of the diagram.
  \[
  \tableau[scY]{\ ,\ ,\ ,\ ,\ ,\tf , \bl \times, \bl \times |
  \ ,\ ,\ ,\ , \tf , \bl \times | 
  \ ,\ ,\ ,\fl , \bl \times |
  \ ,\ ,\fl , \bl \times | 
  \tf , \bl \times, \bl \times}
  \]
  From this, we can see that the column indices of the admissible
  subsets are arbitrary unions of the subsets $\{2\}$, $\{2,3\}$,
  $\{2,3,4\}$, $\{2,3,4,5\}$,
  $\{6\}$, 
  $\{7\}$, and $\{7,8\}$.
\end{eg}

We will need the combinatorics of admissible subsets in one other
setting. Suppose we are given partitions $\alpha$ and $\beta^1 >
\cdots > \beta^r$ such that $(\beta^j, \alpha) \in \VS$ for all $j$,
and $\beta^j \not\subseteq \beta^k$ for $j \ne k$. Suppose also that
$\beta^j / \alpha$ only contains boxes in a single column $c_j$ for
$j=1,\dots,r$. By our assumptions, $c_1 < \cdots < c_r$ and $r \le n$.

From before, we have already defined the sets $\Ad(\alpha; \beta^j)$
for $j=1,\dots,r$. Set $S = \{ (I, (J_i)_{i \in I}) \mid I \subseteq
\{1,\dots,r\},\ J_i \in \Ad(\alpha; \beta^i)\}$.  Note that the $J_i$
may be empty. 
For an element $J = (I, (J_i)_{i \in I}) \in S$, we define
\begin{align*}
  \beta(J) &= \bigcup_{i \in I} \beta^i(J_i), \quad s(J) = \sum_{i \in
    I} (\#J_i + 1).
\end{align*}

We will define the sets $\ol{\Ad}(\alpha; \beta)_i$ and $\Ad(\alpha;
\beta)_i$ by induction on $i$. First, we have $\ol{\Ad}(\alpha;
\beta)_1 = \Ad(\alpha; \beta)_1 = \{ (\{1\}, \emptyset), \dots,
(\{r\}, \emptyset) \}$. In general, let $\ol{\Ad}(\alpha; \beta)_i$ be
the set of $J = (I, (J_i)_{i \in I}) \in S$ such that $\# I = i$ and
such that there does not exist $J' \in \ol{\Ad}(\alpha; \beta)_{i'}$
with $i' < i$ and $\beta(J') = \beta(J)$. We will refer to this last
condition as the {\bf irredundancy condition}. In our setting, this is
equivalent to asking that $\beta(J) / \alpha$ have exactly $s(J)$
columns. Finally, let $\Ad(\alpha; \beta)_i$ denote the admissible
sets $J \in \ol{\Ad}(\alpha; \beta)_i$ such that $\beta(J)$ is a
minimal partition (with respect to inclusion) of the set $\{ \beta(J')
\mid J' \in \ol{\Ad}(\alpha; \beta)_i \}$.

We claim that if $J, J' \in \ol{\Ad}(\alpha; \beta)_i$, then $\beta(J)
\ne \beta(J')$. To see this, let $c_{i_1}$ be the first column index
of $\beta(J) / \alpha$ which is nonempty, and let $J_{i_1} =
\{c_{i_1}, c_{i_1} + 1, \dots, c_{i_1} + k_1\}$ be the longest
consecutive sequence of indices such that the bottom box of column
$c_{i_1} + j$ in $\beta(J)$ is in row $\alpha_{c_{i_1} + j - 1} + 1$
for $j=1,\dots,k_1$. Define $c_{i_2}$ to be the next column index of
$\beta(J) / \alpha$, and define $J_{i_2}$ similarly, etc. We set $J =
(\{i_1, \dots, i_t \}, (J_{i_1}, \dots, J_{i_t}))$. Our procedure
minimizes $t$, which we need to do since we are assuming that $J \in
\ol{\Ad}(\alpha; \beta)$. Furthermore, the choice of indices $\{i_1,
\dots, i_t\}$ uniquely determines the corresponding partitions
$\beta^{j_1}, \dots, \beta^{j_t}$ for which $\beta^{j_k} / \alpha$ is
a single column in column index $i_k$. Namely, we need to take $i_k =
j_k$ for all $k$ by our assumptions on $\alpha$ and $\beta^1, \dots,
\beta^r$, so the claim follows.

This definition will be used in Corollary~\ref{singlecolumns} and in
the proof of Theorem~\ref{r=1}.

\subsection{Pieri resolutions for the general linear
  group.} \label{pieriresolutionsection}

Let $\alpha$ and $\beta^1 > \cdots > \beta^r$ be partitions such that
$(\beta^i, \alpha) \in \VS$ for $i=1,\dots,r$ such that $\beta^i
\not\subseteq \beta^j$ if $i \ne j$. In this section we are concerned
with the minimal free resolution of the cokernel of
\begin{align} \label{cokernelpresentation}
\bigoplus_{i=1}^r A(-|\beta^i / \alpha|) \otimes \Sc_{\beta^i} V \to A
\otimes \Sc_\alpha V,
\end{align}
where the maps are induced by Pieri inclusions. The maps of this form
give presentations of arbitrary equivariant factors of the free module
$A \otimes \Sc_\alpha V$. Let us briefly explain our choice of
assumptions on $\alpha$ and the $\beta^i$. The assumption on the
ordering of the $\beta^i$ is of course harmless as $>$ is a total
order, and the assumption that $(\beta^i, \alpha) \in \VS$ is
necessary to ensure that a nonzero map of the form $A(-|\beta^i /
\alpha|) \otimes \Sc_{\beta^i} V \to A \otimes \Sc_\alpha V$
exists. The assumption $\beta^i \not\subseteq \beta^j$ is made to
eliminate nonminimality: if $\beta^i \subseteq \beta^j$, then the
image of $A(-|\beta^j / \alpha|) \otimes \Sc_{\beta^j} V$ will be
contained in the image of $A(-|\beta^i / \alpha|) \otimes
\Sc_{\beta^i} V$.

We will denote this minimal resolution by ${\bf F}(\alpha;
\beta)_\bullet := {\bf F}(\alpha; \beta^1, \ldots, \beta^r)_\bullet$
and call it a {\bf Pieri resolution} even though we do not yet have
the precise knowledge of its terms.

We first give an inductive procedure for building a free resolution
using just the knowledge of the structure of Pieri resolutions in the
case $r=1$. We use this inductive procedure in
Corollary~\ref{singlecolumns} to give an explicit description in the
case where each $\beta^i / \alpha$ is a single column. Finally we use
this special case to describe explicitly the Pieri resolutions when
$r=1$. 

\begin{theorem} \label{r>1} Let $\alpha$ and $\beta^1 > \cdots >
  \beta^r$ be partitions such that $(\beta^i, \alpha) \in \VS$ for
  $i=1,\dots,r$ and $\beta^i \not\subseteq \beta^j$ for $i \ne j$. An
  equivariant free graded resolution ${\bf F}'(\alpha; \beta)_\bullet$
  of
  \[
  M = \coker \big( \bigoplus_{i=1}^r A(-|\beta^i / \alpha|) \otimes
  \Sc_{\beta^i} V \to A \otimes \Sc_\alpha V \big)
  \]
  can be expressed as an iterated mapping cone of Pieri resolutions
  coming from the case $r=1$. The length of ${\bf F}'(\alpha;
  \beta)_\bullet$ is $\le n+1-c$, where $c$ denotes the index of the
  first column of $\beta^1 / \alpha$.
\end{theorem}

The resolution ${\bf F}'(\alpha; \beta)_\bullet$ may not be minimal,
see Example~\ref{nonminimalexample}. But see
Corollary~\ref{singlecolumns} for a case when it will be minimal. In
general, it will contain ${\bf F}(\alpha; \beta)_\bullet$ as a
subcomplex. 

\begin{proof} We do a double induction, first on $n-c$, and secondly
  on $r$. The base case $n=c$ implies that $r=1$, in which case there
  is nothing to do.

  So suppose $n>c$ and $r>1$. We first note that
  \[
  M = \bigoplus_{\substack{ (\lambda, \alpha) \in \VS \\ (\lambda,
      \beta^i) \notin \VS,\ (1 \le i \le r)}} \Sc_\lambda V
  \]
  (as $\GL(V)$-representations) by Lemma~\ref{acycliclemma}. Define
  \begin{align} \label{Npresentation}
  N = \coker\big(\bigoplus_{i=2}^r \beta^i \to \alpha\big) =
  \bigoplus_{\substack{(\lambda, \alpha) \in \VS \\ (\lambda, \beta^i)
      \notin \VS,\ (2 \le i \le r)}} \Sc_\lambda V.
  \end{align}
  Both $N$ and $M$ are generated over $A$ by $\Sc_\alpha V$. Choosing
  an inclusion $\Sc_\alpha V \to M$, we get a surjection $N \to M$;
  let $N'$ be the kernel of this map. Then $N'$ is the direct sum of
  representations $\Sc_\lambda V$ corresponding to $\lambda$ such that
  $(\lambda, \alpha) \in \VS$, $(\lambda, \beta^1) \in \VS$, and
  $(\lambda, \beta^i) \notin \VS$ for $i=2,\dots,r$. We describe all
  $\lambda$ such that $(\lambda, \beta^1) \in \VS$ which do not appear
  in $N'$.

  If $(\lambda, \beta^1) \in \VS$ and $(\lambda, \alpha) \notin \VS$,
  then $(\lambda, \beta^1(j)) \in \VS$ for some $\{j\} \in \Ad(\alpha,
  \beta^1)$. If $(\lambda, \beta^1) \in \VS$ and $(\lambda, \beta^i)
  \in \VS$ for some $i \ge 2$, then in particular $\lambda \supseteq
  \beta^1 \cup \beta^i$, so $(\lambda, \beta^1 \cup \beta^i) \in
  \VS$. These are the only possibilities, and one can write 
  \[
    N' = \bigoplus_{\substack{ (\lambda, \beta^1) \in \VS \\ (\lambda,
        \beta^1(j)) \notin \VS,\ (j \in \Ad(\alpha, \beta^1)) \\
        (\lambda, \beta^1 \cup \beta^i) \notin \VS,\ (2 \le i \le r)}}
    \Sc_\lambda V = \coker\big( \bigoplus_{i=2}^r (\beta^1 \cup
    \beta^i) \oplus \bigoplus_{\{j\} \in \Ad(\alpha, \beta^1)}
    \beta^1(j) \to \beta^1\big).
  \]
  In fact, in the above presentation, we only need to take those
  partitions of
  \[
  \{\beta^1 \cup \beta^2, \dots, \beta^1 \cup \beta^r\} \cup
  \bigcup_{\{j\} \in \Ad(\alpha, \beta^1)} \beta^1(j)
  \]
  which are minimal with respect to inclusion. So let
  \begin{align} \label{N'presentation}
  N' = \coker\big( \bigoplus (\beta^1 \cup \beta^i) \oplus \bigoplus
  \beta^1(j) \to \beta^1\big)
  \end{align}
  be such a minimal presentation.

  For $N$, the number of relations is $r-1$, so the Pieri resolution
  has been constructed by induction on $r$. For $N'$, the first column
  index of any $(\beta^1 \cup \beta^i) / \beta^1$ or any $\beta^1(j) /
  \beta^1$ is strictly bigger than $c$ because $\beta^1$ is largest in
  lexicographic order, and by definition of critical boxes. Hence the
  Pieri resolution of $N'$ has been constructed by induction on
  $n-c$. Let $(P_\bullet, d)$ and $(P'_\bullet, d')$ be the associated
  Pieri resolutions for the presentation \eqref{Npresentation} of $N$
  and presentation \eqref{N'presentation} of $N'$,
  respectively. Extend the inclusion $f \colon N' \subseteq N$ to an
  equivariant map of resolutions $f_\bullet \colon P'_\bullet \to
  P_\bullet$, and let $F_\bullet$ be the mapping cone of $f_\bullet$.
  We have $F_0 = P_0 = \alpha$ and $F_i = P'_{i-1} \oplus P_i$ for
  $i>0$. The differentials of $F_\bullet$ are
  \[
  \left[ \begin{matrix} -d' & 0 \\ -f_i & d \end{matrix} \right]
  \colon P'_i \oplus P_{i+1} \to P'_{i-1} \oplus P_i,
  \]
  using the convention that $P'_{-1} = 0$, so it is clear that they
  are $\GL(V)$-equivariant. The Pieri resolution of $M$ is a direct
  summand of $F_\bullet$. Writing $P'_\bullet[-1]$ to mean $P'_i[-1] =
  P'_{i-1}$, we have a short exact sequence of chain complexes
  \[
  0 \to P_\bullet \to F_\bullet \to P'_\bullet[-1] \to 0
  \]
  whose long exact sequence of homology shows that $F_\bullet$ is
  acyclic.

  The claim about the length of the resolution follows by induction,
  which gives the easily checked fact that $P'_\bullet$ has length
  $\le n-c$.
\end{proof}


While general Pieri resolutions may have complicated inductive
constructions which involve many cancellations, there are some special
cases when the explicit description can be written down.

\begin{corollary} \label{singlecolumns} Let $\alpha$ and $\beta^1 >
  \cdots > \beta^r$ be partitions with at most $n$ columns such that
  $\beta^i / \alpha$ only contains boxes in the $c_i$th column. Using
  the definitions from Section~\ref{criticalboxessection}, define
  \begin{align*}
    {\bf F}(\alpha; \beta)_0 &= A \otimes \Sc_\alpha V,\\
    {\bf F}(\alpha; \beta)_i &= \bigoplus_{J \in \Ad(\alpha; \beta)_i}
    A(-| \beta(J) / \alpha |) \otimes \Sc_{\beta(J)} V \quad (1 \le i
    \le n+1-c_1).
  \end{align*}
  Then there exist differentials such that ${\bf F}(\alpha;
  \beta)_\bullet$ is a minimal free graded resolution of
  \[
  M = \coker \big( \bigoplus_{i=1}^r A(-|\beta^i / \alpha|) \otimes
  \Sc_{\beta^i} V \to A \otimes \Sc_\alpha V \big).
  \]
\end{corollary}

\begin{proof} We can use the inductive procedure of
  Theorem~\ref{r>1}. First, this description agrees with that of
  Theorem~\ref{equivariantres} in the case that $r=1$ and $c_1=1$,
  which is clear from the definition of admissible subsets. The case
  of general $c_1$ can also be deduced from
  Theorem~\ref{equivariantres}. In particular, suppose that $\beta /
  \alpha$ consists of $m$ boxes in the second column. Let
  $\tilde{\alpha} = (\alpha_2 + m -1, \alpha_2, \alpha_3, \dots,
  \alpha_n)$. Then ${\bf F}(\alpha; \beta)_i = {\bf F}(\tilde{\alpha};
  \alpha)_{i+1}$ for $i \ge 0$. In other words, the resolution is
  obtained by removing the first term of the resolution of Theorem
  \ref{equivariantres}. We leave it to the reader to formulate the
  easy generalization when ``second column'' is replaced by ``$c_1$th
  column.''

  Using the notation from the proof of Theorem~\ref{r>1}, $N$ is
  generated by $\alpha$ with relations $\beta^2, \dots, \beta^r$, so
  satisfies the inductive hypothesis. Also, $N'$ is generated by
  $\beta^1$ with relations $\gamma^1 = \beta^1(c_1+1), \gamma^2 =
  \beta^1 \cup \beta^2, \dots, \gamma^r = \beta^1 \cup \beta^r$. If
  $c_2 > c_1+1$, then these are minimal relations. Otherwise, if $c_2
  = c_1+1$, we throw out $\beta^1(c_1 + 1)$ since it contains $\beta^1
  \cup \beta^2$.  For notation, let $\beta^- = \{\beta^2, \dots,
  \beta^r\}$ so that we can write $\Ad(\alpha; \beta^-)$ in place of
  $\Ad(\alpha; \beta^2, \dots, \beta^r)$, etc.

  By induction, the minimal free resolution of $N'$ is described by
  the admissible sets $\Ad(\beta^1; \gamma)$. The partitions appearing
  in the $k$th term of the resolution of $M$ obtained from the mapping
  cone construction are the partitions of the form $\beta^-(J)$ or
  $\gamma(J')$ for $J \in \Ad(\alpha; \beta^-)_k$ or $J' \in
  \Ad(\beta^1; \gamma)_{k-1}$. So we only need to show that this set
  is equal to $\{\beta(J) \mid J \in \Ad(\alpha; \beta)_k\}$.

  Supposing that this has been done, no two partitions which appear in
  $N$ or $N'$ can be the same because the partitions $\lambda$ of $N'$
  have the property that $\lambda_{c_1} > \alpha_{c_1}$, while those
  of $N$ do not. Hence no cancellations occur, and the mapping cone of
  an equivariant chain map between the Pieri resolutions of $N'$ and
  $N$ lifting the inclusion $N' \subseteq N$ will be the Pieri
  resolution of $M$. This finishes the construction.

  To finish the proof, we establish the promised bijection. First we
  construct a bijection
  \[
  \ol{\phi} \colon \ol{\Ad}(\alpha; \beta)_k \xrightarrow{\cong}
  \ol{\Ad}(\alpha; \beta^-)_k \coprod \ol{\Ad}(\beta^1; \gamma)_{k-1}
  \quad (\text{disjoint union})
  \]
  for $k \ge 1$ such that $\beta(J) = \beta^-(\ol{\phi}(J))$ or
  $\beta(J) = \gamma(\ol{\phi}(J))$ depending on which set
  $\ol{\phi}(J)$ lives in. If $c_2 = c_1 + 1$ so that $\gamma^2,
  \dots, \gamma^r$ are minimal relations for $N'$, then we use subsets
  of $\{2,\dots,r\}$ to index admissible subsets of $\Ad(\beta^1;
  \gamma)$ and $\ol{\Ad}(\beta^1; \gamma)$ for consistency. Pick $J
  \in \ol{\Ad}(\alpha; \beta)$ and write $J = (I, (J_i)_{i \in
    I})$. If $1 \notin I$, then $J \in \ol{\Ad}(\alpha; \beta^-)_i$,
  and we set $\ol{\phi}(J)$ to be this copy of itself.

  If $1 \in I$ and $J_1 \ne \emptyset$, then we set $\ol{\phi}(J) =
  (I, (\ol{\phi}(J)_i)_{i \in I})$ where $\ol{\phi}(J)_1 = J \setminus
  \{c_1+1\}$. If $1 \in I$ and $J_1 = \emptyset$, we set $\ol{\phi}(J)
  = (I \setminus \{1\}, (\ol{\phi}(J)_i)_{i \in I \setminus
    \{1\}})$. In both cases, we set $\ol{\phi}(J)_i = J_i$ for $i>1$,
  so that $\ol{\phi}(J) \in \ol{\Ad}(\beta^1; \gamma)$.

  To check that $\ol{\phi}$ is well-defined, we need to know that
  given $\ol{\phi}(J)$ and $\ol{\phi}(J')$ with $s(\ol{\phi}(J)) <
  s(\ol{\phi}(J'))$ and $\beta^-(\ol{\phi}(J)) =
  \beta^-(\ol{\phi}(J'))$, we have $\beta(J) = \beta(J')$, and we also
  need to show a similar statement with $\beta^-$ replaced by
  $\gamma$. Both of these statements are immediate.

  Also $\ol{\phi}$ is injective by definition. To see that $\ol{\phi}$
  is surjective, we have to check that given $J = (I, (J_i)_{i \in
    I})$ and $J' = (I', (J'_i)_{i \in I'})$ with $\#I < \#I'$ and
  $\beta(J) = \beta(J')$, we have equality of the partitions
  associated with $\ol{\phi}(J)$ and $\ol{\phi}(J')$. If both $I$ and
  $I'$ contain 1 or if neither $I$ nor $I'$ contains 1, then this is
  clear. If only one of $I$ and $I'$ contains 1, we cannot have
  $\beta(J) = \beta(J')$ by our assumption that only $\beta^1 /
  \alpha$ contains boxes in the $c_1$th column.

  Therefore $\ol{\phi}$ has the desired properties. Now we show that
  $\ol{\phi}$ restricts to a bijection
  \[
  \phi \colon \Ad(\alpha; \beta)_k \xrightarrow{\cong} \Ad(\alpha;
  \beta^-)_k \coprod \Ad(\beta^1; \gamma)_{k-1}.
  \]
  First, if $\beta^-(\phi(J)) \subset \beta^-(\phi(J'))$, then we have
  $\beta(J) \subset \beta(J')$ and a similar statement holds when we
  replace $\beta^-$ with $\gamma$. This implies that $\phi$ is
  well-defined.

  Now we establish surjectivity of $\phi$. First pick $\ol{\phi}(J)
  \in \Ad(\alpha; \beta^-)_k$, we have to show that $J \in \Ad(\alpha;
  \beta)_k$. Suppose not, so that there exists $J' \in \Ad(\alpha;
  \beta)_k$ with $\beta(J') \subsetneqq \beta(J)$. Since $\beta(J) /
  \alpha$ contains no boxes in the $c_1$th column, the same is true
  for $\beta(J') / \alpha$, so we know that $\ol{\phi}(J') \in
  \ol{\Ad}(\alpha; \beta^-)_k$, and that $\beta^-(\ol{\phi}(J'))
  \subsetneqq \beta^-(\ol{\phi}(J))$, a contradiction. 

  Now pick $\ol{\phi}(J) \in \Ad(\beta^1; \gamma)_{k-1}$, we have to
  show that $J \in \Ad(\alpha; \beta)_k$. Again, suppose not so that
  there exists $J' \in \Ad(\alpha; \beta)_k$ such that $\beta(J')
  \subsetneqq \beta(J)$. Write $J = (I, (J_i)_{i \in I})$ and $J' =
  (I', (J'_i)_{i \in I'})$. We know that $1 \in I$. If $1 \in I'$,
  then $\ol{\phi}(J') \in \ol{\Ad}(\beta^1; \gamma)$, and we get a
  contradiction as before. So suppose that $1 \notin I'$. By the
  irredundancy condition in the definition of $\ol{\Ad}$, we have that
  $s(J)$ is the number of columns of $\beta(J) / \alpha$, and
  similarly for $J'$. Hence $s(J) > s(J')$, which is a
  contradiction. We conclude that $J \in \Ad(\alpha; \beta)_k$, so
  $\phi$ is surjective.  Injectivity of $\phi$ is immediate from
  injectivity of $\ol{\phi}$, so we have the desired bijection.
\end{proof}

Now we specialize to the case that $r=1$, and write $\beta =
\beta^1$. Define equivariant $A$-modules ${\bf F}(\alpha;
\beta)_\bullet$ as follows:
\begin{align*} 
  {\bf F}(\alpha; \beta )_0 &= A \otimes \Sc_\alpha V,\\
  {\bf F}(\alpha; \beta )_i &= \bigoplus_{J \in \Ad(\alpha; \beta)_i}
  A(-|\beta(J) / \alpha|) \otimes \Sc_{\beta(J)} V, \quad (0 < i \le
  n).
\end{align*}
Theorem~\ref{r=1} will show that one can choose equivariant
differentials so that ${\bf F}(\alpha; \beta)_\bullet$ becomes a
minimal free resolution. Note that ${\bf F}(\alpha; \beta)_1 =
A(-|\beta / \alpha|) \otimes \Sc_{\beta} V$, and that these modules
agree with those defined in Section~\ref{purefreesection} in the case
that $\beta/\alpha$ consists of just boxes in the first column.

\begin{remark}
  The definition of ${\bf F}(\alpha; \beta)_\bullet$ is natural from
  the following point of view. The idea is to consider which
  representations appear in the kernel of the map
  \[
  A(-|\beta/\alpha|) \otimes \Sc_{\beta} V \to A \otimes \Sc_\alpha V,
  \]
  i.e., those representations of highest weight $\lambda$ where
  $(\lambda, \alpha) \notin \VS$, and $(\lambda, \beta) \in \VS$. Then
  we find a minimal generating set of such representations, and then
  surject onto them using other representations. This explains the
  ${\bf F}(\alpha; \beta)_2$ term, and we continue in this way; the
  language of critical boxes and admissible sets is a convenient way
  to describe such minimal partitions.
\end{remark}

\begin{theorem} \label{r=1} Let $\alpha$ and $\beta$ be partitions
  with at most $n$ columns such that $\beta \supset \alpha$ and $\beta
  / \alpha$ is a vertical strip. Then there exist equivariant
  differentials ${\bf F}(\alpha; \beta)_{i+1} \to {\bf F}(\alpha;
  \beta)_i$ making ${\bf F}(\alpha; \beta)_\bullet$ a minimal free
  graded resolution of $M = \coker(\phi(\alpha, \beta))$. Furthermore,
  the length of ${\bf F}(\alpha; \beta)_\bullet$ is $n+1-c$, where $c$
  denotes the index of the first column of $\beta / \alpha$.
\end{theorem}

\begin{proof} First note that $M$ is the direct sum (as $\GL(V)$
  representations) of $\Sc_\lambda V$ over all $\lambda$ such that
  $(\lambda, \alpha) \in \VS$ and $(\lambda, \beta) \notin \VS$. This
  follows from Lemma~\ref{acycliclemma}. Write $\beta = \alpha +
  a_{i_1}e_{i_1} + \cdots + a_{i_{b}}e_{i_{b}}$ to mean that $\beta$
  is obtained from $\alpha$ by adding $a_{i_j}$ boxes in column
  $i_j$. We proceed by induction on $b$. The case $b=1$ follows from
  Corollary~\ref{r>1}.

  Now assume $b > 1$. Set $\alpha' = \alpha + a_{i_1}e_{i_1} + \cdots
  + a_{i_{b-1}}e_{i_{b-1}}$ and define
  \[
  N = \bigoplus_{\substack{(\lambda, \alpha') \in \VS\\ (\lambda,
      \beta) \notin \VS \\ (\lambda, \alpha'(i_j + 1)) \notin \VS,\ (1
      \le j \le b-1) }} \Sc_\lambda V,
  \]
  which is an $A$-submodule of $M$. For clarity, we will abuse
  notation and abbreviate $A(-|\beta(J) / \alpha|) \otimes
  \Sc_{\beta(J)} V$ by simply $\beta(J)$. The minimal free resolution
  $Q_\bullet$ of $N$ has been constructed in
  Corollary~\ref{r>1}. Namely, set $\gamma^j = \alpha'(i_j + 1)$ for
  $1 \le j \le b-1$ and $\gamma^b = \beta$. We have $\gamma^{b}
  \subset \gamma^{b-1}$ if and only if $i_b = i_{b-1} + 1$, and there
  are no other inclusions. Let $\gamma = \{\gamma^1, \dots,
  \gamma^{b-2}, \gamma^b\}$ in the first case, and $\gamma =
  \{\gamma^1, \dots, \gamma^b\}$ in the second case. Then $Q_k =
  \bigoplus_{J \in \Ad(\alpha'; \gamma)_k} \gamma(J)$, where we are
  adopting the shorthand above.

  The quotient $N' = M/N$ is the direct sum (as a representation of
  $\GL(V)$; $N'$ is not an $A$-submodule of $M$ in general)
  \[
  N' = \bigoplus_{\substack{(\lambda, \alpha) \in \VS\\ (\lambda,
      \alpha') \notin \VS}} \Sc_\lambda V,
  \]
  which is also the cokernel of a Pieri inclusion
  $A(-|\alpha'/\alpha|) \otimes \Sc_{\alpha'} V \to A \otimes
  \Sc_\alpha V$. The minimal free resolution $Q'_\bullet$ of $N'$ is
  \[
  0 \to \DS \bigoplus_{J \in \Ad(\alpha; \alpha')_{n-c+1}} \alpha'(J)
  \to \cdots \to \DS \bigoplus_{J \in \Ad(\alpha; \alpha')_2}
  \alpha'(J) \to \alpha' \to \alpha,
  \]
  which follows by induction since $\alpha' / \alpha$ has $b-1$
  columns. Using the short exact sequence
  \[
  0 \to N \to M \to N' \to 0,
  \]
  we can construct a $\GL(V)$-equivariant resolution
  $\tilde{P}_\bullet$ of $M$ whose terms are coordinatewise direct
  sums of the resolutions of $N$ and $N'$.
  So the partitions generating the module $\tilde{P}_k$ are
  \[
  \{\alpha'(J) \mid J \in \Ad(\alpha; \alpha')_k\} \coprod \{\gamma(J)
  \mid J \in \Ad(\alpha'; \gamma)_k\}.
  \]
  However, $\tilde{P}_\bullet$ is not a minimal resolution of $M$. In
  particular, any partition that appears as a generator of a module in
  the minimal resolution of $M$ (and which is not $\alpha)$ must
  contain $\beta$. The terms generated by partitions which do not
  contain $\beta$ form a subcomplex of $\tilde{P}_\bullet$. After a
  change of basis, we can show that the positive degree terms of this
  subcomplex (so excluding the degree 0 term generated by $\alpha$) is
  also a subcomplex. Let $P_\bullet$ be the quotient of
  $\tilde{P}_\bullet$ by this subcomplex. Then $P_\bullet$ contains as
  generators only $\alpha$ in degree 0 and those generators of
  $\tilde{P}_\bullet$ which contain $\beta$. We claim that $P_\bullet$
  is the minimal resolution of $M$.

  To show this, it is enough to show that the terms of $P_\bullet$
  agree with the terms of ${\bf F}(\alpha; \beta)_\bullet$ since the
  generating partitions of the latter are all distinct. So we will be
  finished if we can establish a bijection 
  \[
  \psi \colon \Ad(\alpha; \beta)_k \xrightarrow{\cong} \{ J \in
  \Ad(\alpha; \alpha')_k \mid \alpha'(J) \supseteq \beta \} \coprod
  \{J \in \Ad(\alpha'; \gamma)_k \mid \gamma(J) \supseteq \beta \}
  \]
  for all $k$. So pick $J \in \Ad(\alpha; \beta)_k$. We can uniquely
  write $J = \{i_{j_1} + 1, \dots, i_{j_1} + k_1\} \cup \cdots \cup \{
  i_{j_t} + 1, \dots, i_{j_t} + k_t\}$ so that $i_r + 1 \notin
  [i_{j_s} + 1, i_{j_s} + k_s]$ unless $r = s$. Then $k_1 + \cdots +
  k_t = k -1$. If $i_b \in J$, then $J \in \Ad(\alpha; \alpha')_k$ and
  $\alpha'(J) \supseteq \beta$, so we define $\psi(J)$ to be this copy
  of $J$. 

  Otherwise, if $i_b \notin J$, we first define $I = \{j_1, \dots,
  j_t, b\}$ and $\psi(J)_{j_s} = \{i_{j_s} + 2, \dots, i_{j_s} +
  k_s\}$ for $s=1,\dots, t-1$, and do the same for $s=t$ if $b \ne
  j_t$. If $b \ne j_t$, set $\psi(J)_b = \emptyset$, and if $b = j_t$,
  set $\psi(J)_b = \{i_{j_t} + 1, \dots, i_{j_t} + k_t\}$. Now set
  $\psi(J) = (I, (\psi(J)_i)_{i \in I})$, which we claim is an element
  of $\Ad(\alpha'; \gamma)_k$. To see that it is an element of
  $\ol{\Ad}(\alpha'; \gamma)_k$, it is enough to show that
  $\gamma(\psi(J)) / \alpha'$ has exactly $k$ columns. This is true
  because the column indices of $\gamma(\psi(J)) / \alpha'$ are
  precisely $J \cup \{i_b\}$. Now since $\gamma(\psi(J)) = \beta(J)$,
  if $J' \in \ol{\Ad}(\alpha'; \gamma)_k$ existed so that $\gamma(J')
  \subsetneqq \gamma(\psi(J))$, one could use $J'$ to find a
  corresponding $J'' \in \ol{\Ad}(\alpha; \beta)$ so that $\beta(J'')
  \subsetneqq \beta(J)$. The argument is similar to the one found in
  the proof of Corollary~\ref{r>1}, so we omit the details.

  The above establishes that $\psi$ is well-defined, and it is
  injective by construction. To finish, we show that $\psi$ is also
  surjective. First pick $J \in \Ad(\alpha; \alpha')_k$ such that
  $\alpha'(J) \supseteq \beta$. Then $i_b \in J$ and $J$ can also be
  considered an element of $\Ad(\alpha; \beta)_k$, so we're done in
  this case. Now pick $J \in \Ad(\alpha'; \gamma)_k$ and write $J =
  (I, (J_i)_{i \in I})$. By the irredundancy condition in the
  definition of $\Ad$, it follows that $J' = \{i_s \mid s \in I\} \cup
  \bigcup_{i \in I} J_i$ is a set of size $k$. So $J' \setminus \{i_b\}
  \in \Ad(\alpha; \beta)_k$ and $\psi(J' \setminus \{i_b\}) = J$.

  Finally, by definition the length of ${\bf F}(\alpha;
  \beta)_\bullet$ is the size $s(J)$ of the largest admissible set
  $J$. Every admissible set is a subset of $\{c+1, c+2, \dots, n\}$,
  and this set is admissible, so the second statement of the theorem
  follows.
\end{proof}

There is a natural family of resolutions ${\bf F}(\alpha;
\beta)_\bullet$ which are pure.

\begin{corollary} If $\beta/\alpha$ contains boxes only in the $i$th
  column and the $n$th column for some $i$, then the complex ${\bf
    F}(\alpha; \beta )_\bullet$ is pure.
\end{corollary}

\begin{rmk}[Pieri resolutions over $\bigwedge V$] \rm Following
  Remark~\ref{exteriorpurefree}, it is not hard to see how to
  construct Pieri resolutions over the exterior algebra $B$. The only
  thing that changes is that columns and rows are swapped in the
  notion of critical boxes, and the resulting resolution will have
  infinite length and is eventually linear.
\end{rmk}

\begin{remark} In principle, one can resolve a general finitely
  generated equivariant module $M$. Pick any cyclic $A$-submodule $N$
  of $M$. Then $M/N$ has less generators than $M$, and resolutions for
  $N$ and $M/N$ can be combined to give a resolution of $M$. We do
  not attempt to resolve cokernels of maps of the form
  \[
  \bigoplus_{i=1}^r A \otimes \Sc_{\beta^i} V \to \bigoplus_{j=1}^g A
  \otimes \Sc_{\alpha^j} V
  \]
  simply because the cokernel is not uniquely defined by the domain
  and codomain. It is not even enough to specify which maps $A \otimes
  \Sc_{\beta^i} V \to A \otimes \Sc_{\alpha^j} V$ are nonzero since
  for $r > 1$ and $g > 1$, the scalar multiples of the Olver maps used
  becomes important.
\end{remark}

\subsection{Examples.}\label{examplesection}

\begin{remark} \label{trivialmodule} Given an irreducible
  representation $\Sc_\alpha V$, we can think of it as a trivial
  $A$-module, i.e., its annihilator is the maximal ideal $(x_1, \dots,
  x_n)$. Then a minimal presentation of $\Sc_\alpha V$ is obtained as
  follows. Let $1 = i_1 < \cdots < i_r \le n$ denote the indices such
  that $\alpha_{i_j} < \alpha_{i_j-1}$ for $j=2,\dots,r$. Let $\beta^j
  = (\alpha_1, \dots, \alpha_{i_j}+1, \dots, \alpha_n)$, so that
  \[
  \Sc_\alpha V = \coker \big( \bigoplus_{i=1}^r A(-1) \otimes
  \Sc_{\beta^i} V \to A \otimes \Sc_\alpha V),
  \]
  and one can write down its minimal resolution using
  Corollary~\ref{singlecolumns}. 

  In particular, if $\alpha = (0, \dots, 0)$, then $\Sc_\alpha V = K$,
  and the Pieri resolution in this case is the Koszul complex. In
  general, this resolution is $\Sc_\alpha V$ tensored with the Koszul
  complex on all of the variables $x_1, \dots, x_n$.
\end{remark}

\begin{eg} Let $n = 4$, $\alpha = (5,3,1,0)$, and $\beta =
  (6,4,1,0)$. The corresponding Young diagram is
  \[
  \tableau[scY]{,,|,,\bl, \bl \times|,|,1, \bl \times||1,\bl \times}
  \]
  The critical boxes are marked with $\times$, and our Pieri
  resolution is
  \[
  0\to (6,6,4,2)\to (6,6,4,0) \oplus (6,4,4,2)\to (6,6,1,0) \oplus
  (6,4,4,0)\to (6,4,1,0)\to (5,3,1,0).
  \]
\end{eg}

\begin{eg} Here is an example of a pure resolution of new type: $n=4$,
  $\alpha = (5,3,1,0)$, $\beta = (6,3,1,1)$. The corresponding Young
  diagram is
  \[
  \tableau[scY]{,,,1 | , ,\bl , \bl \times| , | ,\bl ,\bl \times | |
    1,\bl \times}
  \]
  The critical boxes are marked with $\times$. The Pieri resolution
  ${\bf F}(\alpha; \beta)_\bullet$ we get is
  \[
  0\to (6,6,4,2)\to (6,6,4,1)\to (6,6,1,1)\to (6,3,1,1)\to (5,3,1,0).
  \]
  The length of the complex is 4, but the module it resolves has
  infinite length (the cokernel contains all highest weight modules
  for partitions $(5+d,3,1,0)$ for $d\ge 0$). Thus this module is not
  a maximal Cohen--Macaulay module.
\end{eg}

\begin{eg} \label{singlecolumnsexample} To illustrate
  Corollary~\ref{singlecolumns}, let $\alpha = (4,3,1,0)$, $\beta^1 =
  (6,3,1,0)$, and $\beta^2 = (4,3,3,0)$. Then the column indices of
  the admissible subsets of $C(\alpha, \beta^1)$ and $C(\alpha,
  \beta^2)$ are, respectively, $\{\{2\}$, $\{2,3\}$, $\{2,3,4\}\}$ and
  $\{4\}$. In homological degree 2, the candidates for syzygy
  generators are $\{ \beta^1(2), \beta^1 \cup \beta^2, \beta^2(4)\}$,
  in degree 3, they are $\{ \beta^1(2,3), \beta^1(2) \cup \beta^2,
  \beta^1 \cup \beta^2(4)\}$, and in degree 4, they are $\{
  \beta^1(2,3,4), \beta^1(2,3) \cup \beta^2, \beta^1(2) \cup
  \beta^2(4)\}$.

  Since $\beta^1(2,3) \cup \beta^2 = (6,5,4,0) = \beta^1(2,3)$, we
  remove it from our list of candidates in degree 4. Finally, we pick
  only those candidates which are minimal in their homological degree
  with respect to inclusion. The resolution is then
  \[
  0 \to (6,5,3,2) \to \begin{array}{c} (6,5,3,0)\\
    (6,3,3,2) \end{array} \to \begin{array}{c} (6,5,1,0)\\ (6,3,3,0)\\
    (4,3,3,2) \end{array} \to \begin{array}{c} (6,3,1,0)\\
    (4,3,3,0)\end{array} \to (4,3,1,0).
  \]
  Here we are stacking partitions as shorthand for the direct sum of
  the free $A$-modules generated by the corresponding representations.
\end{eg}

\begin{eg} \label{multiplicityexample} Here is an example to show that
  the representations that appear in a Pieri resolution may have
  nontrivial multiplicities. Let $n=3$, $\alpha = (3,1,0)$, $\beta^1 =
  (4,3,0)$, $\beta^2 = (3,3,1)$, and $\beta^3 = (4,2,1)$. The Pieri
  resolution is
  \[
  0 \to \begin{array}{c} (4,4,1) \\ (4,3,2) \end{array}
  \to \begin{array}{c} (4,4,0) \\ (4,3,1) \\ (4,3,1) \\ (4,2,2) \\
    (3,3,2) \end{array} \to \begin{array}{c} (4,3,0) \\
    (4,2,1) \\ (3,3,1) \end{array} \to (3,1,0).
  \]
  Again, we use the same shorthand from
  Example~\ref{singlecolumnsexample}. Note that the partition
  $(4,3,1)$ appears twice as a generator. This is reasonable: the
  representation $(4,3,1)$ appears only once in homological degree 0,
  but appears 3 times in homological degree 1. Also note that this
  resolution is pure (though not of a Cohen--Macaulay module).
\end{eg}

\begin{eg} \label{nonminimalexample} Here is an example to show that
  the mapping cone construction of Theorem~\ref{r>1} need not return a
  minimal resolution. Let $n=4$, $\alpha = (4,2,1,0)$, $\beta^1 =
  (5,3,1,0)$ and $\beta^2 = (5,2,2,0)$. Working out the induction, we
  get the following resolution
  \[
  0 \to \begin{array}{c} (5,5,3,2) \\ (5,5,2,2) \end{array} \to 
  \begin{array}{c} (5,5,3,0) \\ (5,5,2,2) \\ (5,5,2,0) \\
    (5,3,2,2) \end{array} \to 
  \begin{array}{c} (5,5,2,0) \\ (5,5,1,0) \\ (5,3,2,0) \\
    (5,2,2,2) \end{array} \to
  \begin{array}{c} (5,3,1,0) \\ (5,2,2,0) \end{array} \to
  (4,2,1,0).
  \]
  The actual calculation can be done in Macaulay 2 using the package
  {\tt PieriMaps} \cite{pierimaps}, and one gets the following graded
  Betti table:
  \[
  \begin{matrix}
    &0&1&2&3&4\\
    \text{total:}&140&520&600&269&60\\
    0:&140&.&.&.&.\\
    1:&.&520&300&.&.\\
    2:&.&.&300&45&.\\
    3:&.&.&.&224&.\\
    4:&.&.&.&.&60
  \end{matrix}.
  \]
  In particular, the differentials between the pair of representations
  $(5,5,2,2)$ and the pair of representations $(5,5,2,0)$ are
  isomorphisms. The minimal resolution of the cokernel is
  \[
  0 \to (5,5,3,2) \to
  \begin{array}{c} (5,5,3,0) \\ (5,3,2,2) \end{array} \to 
  \begin{array}{c} (5,5,1,0) \\ (5,3,2,0) \\ (5,2,2,2) \end{array} \to
  \begin{array}{c} (5,3,1,0) \\ (5,2,2,0) \end{array} \to
  (4,2,1,0).
  \]
\end{eg}

\begin{remark} For the sequence $e = (e_0, \ldots, e_n)$ (with $e_0
  =0$) we can produce $e_1$ pure complexes with shifts $e$
  (corresponding to choices of $i$ boxes in the first column and $e_1
  - i$ boxes in the last column, for $1\le i\le e_1$). For example,
  take $n=4$, $e = (0,3,4,2,1)$. There are three pure complexes in our
  family, corresponding to the Young diagrams
  \[
  \tableau[scY]{,,3,4 | ,2,3 | ,2 | ,2 | 1,2 | 1 | 1 } \quad
  \tableau[scY]{,,,1 | ,,3,4 | ,2,3 | ,2 | ,2 | 1,2 | 1} \quad
  \tableau[scY]{,,,1 | ,,,1 | ,,3,4 | ,2,3 | ,2 | ,2 | 1,2 }.
  \]
\end{remark}

\section{Equivariant resolutions for other classical
  groups.} \label{otherclassicalgroups}


In this section we generalize the resolutions from \cite{efw} to other
classical groups in the following way. We use the vector
representation $F$ of an orthogonal or a symplectic group. Denote by
$V_\lambda$ the irreducible representation of the corresponding
classical group whose highest weight is $\lambda$. The highest weights
in each case correspond to partitions with some restrictions (see
Table~\ref{weighttable} below). Consider the polynomial ring $A =
\Sym(F)$. For a pair of representations $V_\lambda$, $V_\mu$ such that
$V_\mu \subset V_\lambda \otimes \Sym^i F$ we can ask again about the
resolution of the cokernel $M$ of the Pieri map $V_\mu \otimes A(-i)
\to V_\lambda \otimes A$. Using the Pieri resolutions discussed in the
previous sections and some sheaf cohomology, we construct resolutions
for a certain quotient $N$ of this cokernel in the case that $|\mu| =
|\lambda| + i$ (see Remark~\ref{classicalpierirule}). The comments at
the end of Section~\ref{section:typeB} explain the relationship
between the resolutions of $M$ and $N$. More generally, we can also
study direct sums of Pieri maps. We consider the cases of odd
orthogonal, symplectic, and even orthogonal groups separately.

This section is independent of Section~\ref{equivariantbssection}, so
the reader not comfortable with the representation theory of classical
groups can skip this section without any loss of
continuity.

\subsection{Notation.}

Let $F$ be a $(2n+\tau)$-dimensional vector space over $K$ (where
$\tau \in \{0,1\}$) and let $\omega$ be a nondegenerate symplectic or
symmetric bilinear form such that we can find a basis $e_1, \dots,
e_{2n+\tau}$ for $F$ satisfying $1 = \omega(e_i, e_{2n+\tau+1-i}) =
\pm \omega(e_{2n+\tau+1-i}, e_i)$ (the sign depending on whether
$\omega$ is symmetric or skew-symmetric) for $i=1,\dots,n+\tau$, and
all other pairings are 0. Let $G$ be the subgroup of $\SL(F)$ which
preserves $\omega$. In order to be precise let us just list the cases:
\begin{compactenum}
\item Case $\mathrm{B}_n$: We have $\tau = 1$, $\omega$ is symmetric,
  $G = \SO(F) \cong \SO(2n+1)$.
\item Case $\mathrm{C}_n$: We have $\tau = 0$, $\omega$ is
  skew-symmetric, $G = \Sp(F) \cong \Sp(2n)$.
\item Case $\mathrm{D}_n$: We have $\tau = 0$, $\omega$ is symmetric,
  $G = \SO(F) \cong \SO(2n)$.
\end{compactenum}

We identify the weight lattice of $G$ with $\Z^n = \Z\langle \eps_1,
\dots, \eps_n \rangle$ equipped with the standard dot product. The
identification is crucial for our applications of
Theorem~\ref{bottstheorem}, so we make this more precise. Representing
elements of $G$ as matrices in the ordered basis $\{e_1, \dots,
e_{2n+\tau}\}$, we take our maximal torus $T$ to be the subgroup of
diagonal matrices, and our Borel subgroup $B$ to be the subgroup of
upper triangular matrices, so that we have a set of simple roots for
$G$. We identify $(\lambda_1, \dots, \lambda_n) \in \Z^n$ with the
character $\lambda \colon T \to K^*$ given by
\[
\operatorname{diag}(d_1, \dots, d_{2n+\tau}) \mapsto \prod_{i=1}^n
d_i^{\lambda_i} 
\]
To be completely explicit, we list the simple roots and the conditions
for a weight to be dominant under this identification in
Table~\ref{weighttable}. We choose this particular identification to
be compatible with the identification of weights for $\GL_n$. In
particular, whenever we have an $n$-tuple $\lambda$, the statement
that $\lambda$ is dominant for $\GL_n$ will mean that $\lambda$ is a
weakly decreasing sequence, and the statement that $\lambda$ is
dominant for any other classical group $G$ will mean that it satisfies
the appropriate condition according to Table~\ref{weighttable}.
\begin{table}[ht]
\caption{Roots and weights}
\label{weighttable}
\centering
\begin{tabular}{|c|c|c|c|}
  \hline
  & $\SO(2n+1)$ & $\Sp(2n)$ & $\SO(2n)$ \\
  \hline 
  \tiny{simple roots} & 
  \begin{tabular}{l}
    $\eps_1 - \eps_2, \eps_2 - \eps_3,$ \\ 
    $\dots, \eps_{n-1} - \eps_n, \eps_n$ \end{tabular} & 
  \begin{tabular}{l}
    $\eps_1 - \eps_2, \eps_2 - \eps_3,$ \\ 
    $\dots, \eps_{n-1} - \eps_n, 2\eps_n$ \end{tabular} & 
  \begin{tabular}{l}
    $\eps_1 - \eps_2, \eps_2 - \eps_3,$ \\ 
    $\dots, \eps_{n-1} - \eps_n, \eps_{n-1} + \eps_n$ \end{tabular}
  \\
  \hline
  \tiny{dominant weights} & $ \lambda_1 \ge \cdots \ge \lambda_n \ge 0 $ & 
  $\lambda_1 \ge \cdots \ge \lambda_n \ge 0$ & $\lambda_1 \ge
  \cdots \ge \lambda_{n-1} \ge |\lambda_n|$ \\
  \hline
  $\rho$ & $(\frac{2n-1}{2}, \frac{2n-3}{2}, \dots, \frac{1}{2})$ &
  $(n, n-1, \dots, 2, 1)$ & $(n-1, n-2, \dots, 1, 0)$ \\
  \hline
\end{tabular}
\end{table}

For $\lambda$ a dominant weight of $G$, the notation $V_\lambda$
denotes an irreducible representation of $G$ with highest weight
$\lambda$. For $\lambda = (k,0,\dots,0)$, we just write $V_k$. In the
case of the symplectic group, $V_k = \Sym^k F$, but this is false in
the orthogonal case: in fact, $V_k$ is the kernel of the contraction
map $\Sym^k F \to \Sym^{k-2} F$ given by $x_1 \cdots x_k \mapsto
\sum_{i < j} \omega(x_i, x_j) x_1 \cdots \hat{x_i} \cdots \hat{x_j}
\cdots x_k$.

\begin{remark} \label{classicalpierirule} To give the reader a sense
  of why resolving equivariant modules over orthogonal and symplectic
  groups might be more difficult than the case of the general linear
  group, we recall the corresponding analogues of Pieri's rule.

  Let $G$ be an orthogonal or symplectic group of rank $n$. The
  Newell--Littlewood rule (see \cite[\S 4]{king}) gives, under the
  assumption that $n \ge \ell(\lambda) + \ell(\mu)$, that $V_\lambda
  \otimes V_{\mu} = \bigoplus_\nu V_\nu^{\oplus N^\nu_{\lambda, \mu}}$
  where $N^\nu_{\lambda, \mu} = \sum_{\alpha, \beta, \gamma}
  c^\lambda_{\alpha, \beta} c^\mu_{\alpha, \gamma} c^\nu_{\beta,
    \gamma}$ and $c_{*,*}^*$ denotes the usual Littlewood--Richardson
  coefficient for the general linear group. 

  Now specialize to the case that $\mu = (k)$, so that $c^\mu_{\alpha,
    \gamma} \ne 0$ in the above sum only if $\alpha$ and $\gamma$ are
  vertical strips of sizes $l$ and $k-l$ for some $l \le k$.  Hence,
  if $\lambda$ is a partition such that $\lambda_n = 0$, then
  $V_\lambda \otimes V_k \cong \bigoplus_\mu V_\mu^{\oplus n_\mu}$
  where the sum is over all $\mu$ which can be obtained from $\lambda$
  by first removing a vertical strip of size $l \le k$, and then
  adding a horizontal strip of size $k-l$ to the result, and $n_\mu$
  is the number of different ways to obtain $\mu$ via this process.

  The above formulas can still be interpreted when $\lambda_n \ne 0$
  or $n < \ell(\lambda) + \ell(\mu)$. In this case, one needs to use
  certain modification rules to rewrite $V_\nu$ where $\ell(\nu) > n$
  as $V_{\eta}$ where $\ell(\eta) \le n$. See \cite[\S 3]{king} and
  \cite[\S\S 2.4--2.5]{youngdiagram} for more details.

  Hence it is not clear how to give a combinatorial description of the
  terms in an equivariant resolution since in general, the same
  representation may appear as a generator for several different
  syzygy modules. Furthermore, the Pieri rule for orthogonal and
  symplectic groups is not multiplicity free. $\blacksquare$
\end{remark}

Let $X' = \{V \in \mathbf{Gr}(n, F) \mid \omega|_V = 0\}$ be the
Grassmannian of Lagrangian subspaces of $F$. Let $x_0 \in X'$ be the
point representing the subspace $\langle e_1, \dots, e_n \rangle$, and
let $X$ be the connected component of $X'$ which contains $x_0$. Then
$X$ is a homogeneous space for $G$, and we can identify $X$ with $G/P$
where $P$ is the parabolic subgroup which stabilizes $x_0$. On $X$, we
have the tautological subbundle $\RR$ of the trivial bundle $F \times
X$ defined by $\RR = \{(x, W) \mid x \in W\}$, and also the
tautological quotient bundle $\ol{\QQ} = (F \times X) / \RR$. Let
$\RR^\vee$ denote the orthogonal complement of $\RR$ with respect to
$\omega$, and define $\QQ = (F \times X) / \RR^\vee$. Since we have
perfect pairings $\omega \colon \RR \times \QQ \to K \times X$ and
$\omega \colon \RR^\vee \times \ol{\QQ} \to K \times X$, the form on
$F$ gives identifications $\QQ^* = \RR$ and $\ol{\QQ}^* =
\RR^\vee$. When $F$ is even-dimensional, $\RR = \RR^\vee$, and $\QQ =
\ol{\QQ}$.

We shall not make a distinction between vector bundles and locally
free sheaves, so that $\RR, \QQ, \dots$ will be a sheaf when we want
to calculate its cohomology, and will be a vector bundle when we need
to work with its total space.

There is an equivalence between homogeneous bundles over $X = G/P$ and
rational representations of $P$ defined by sending a homogeneous
bundle $\EE$ to its fiber over the point $x_0$, which represents the
coset $P$. In the other direction, given a rational $P$-module $U$, we
define a homogeneous bundle
\[
G \times^P U = G \times U / \{(g,u) \sim (gp, p^{-1}u)\ (p \in P)\}
\]
with $G$ acting on the first factor by multiplication on the left, and
the structure map $G \times^P U \to G/P$ is given by $(g,u) \mapsto
gP$. Under this equivalence, $\RR$ is associated with an isotropic
subspace $R \subset F$, and $\QQ$ to the quotient $Q =
F/R^\vee$. Hence, $\Sc_\lambda \QQ$ is a homogeneous bundle over $X$,
and we will need to know something about its cohomology.

Let $W$ be the Weyl group of $G$, and let $\ell(\sigma)$ denote the
length of $\sigma \in W$. We define a dotted action of $W$ on the
weights of $G$ by $\sigma^\bullet(\lambda) = \sigma(\lambda + \rho) -
\rho$, where $\rho$ is given in Table~\ref{weighttable}.

\begin{theorem} \label{bottstheorem} With the notation above, one of
  two mutually exclusive cases occurs:
  \begin{compactenum}
  \item There exists $\sigma \in W$ such that $\sigma^\bullet(\alpha)
    = \alpha$. In this case, $\Hs^i(G/P; \Sc_\alpha \QQ) = 0$ for all
    $i$.
  \item There exists a unique $\sigma \in W$ such that $\beta =
    \sigma^\bullet(\alpha)$ is a dominant weight for $G$. Then
    \[
    \Hs^i(G/P; \Sc_\alpha \QQ) = \begin{cases} V^*_\beta &
      \text{if } i = \ell(\sigma),\\
      0 & \text{otherwise.}
      \end{cases}
    \]
    \end{compactenum}
\end{theorem}

\begin{proof} This theorem is a special case of the Borel--Weil--Bott
  theorem for the homogeneous space $G/P$. See \cite[\S II.5]{jantzen}
  for a proof of the Borel--Weil--Bott theorem, and see \cite[\S
  4.3]{weyman} for more details regarding our special case. It should
  be mentioned that in \cite[\S 4.3]{weyman} there is an error: the
  $V_\alpha$ should be $V^*_\alpha$ in Theorem 4.3.1, and similarly
  for the other results in that section.
\end{proof}

\begin{remark} \label{dualweights} In types ${\rm B}_n$ and ${\rm
    C}_n$, all representations are self-dual, so that $V^*_\beta \cong
  V_\beta$. The same is true for type ${\rm D}_n$ when $n$ is
  even. However, when $n$ is odd, the dual of a representation of type
  ${\rm D}_n$ with highest weight $(\beta_1, \dots, \beta_{n-1},
  \beta_n)$ has highest weight $(\beta_1, \dots, \beta_{n-1},
  -\beta_n)$.
\end{remark}

\subsection{General setup.}

Now define $\BB = \Sym(\QQ)$. Let $\alpha$ and $\beta^1, \dots,
\beta^r$ be partitions with at most $n$ parts such that
$\beta^i/\alpha$ is a vertical strip for $i=1,\dots,r$. Pieri's rule
extends to the setting of vector bundles, so the sheaf $\Sc_\alpha \QQ
\otimes \BB$ contains $\Sc_{\beta^i} \QQ$ as a direct summand with
multiplicity one, and hence we have a unique (up to nonzero scalar)
$\GL(F)$-equivariant morphism of sheaves
\[
\bigoplus_{i=1}^r \BB \otimes \Sc_{\beta^i} \QQ \to \BB \otimes
\Sc_\alpha \QQ.
\]
We can resolve the cokernel of this map via the relative version of
the minimal Pieri resolution, ${\bf F}_\bullet := {\bf F}(\alpha;
\beta)_\bullet$ defined in Section~\ref{equivariantressection} (see
the remarks before Theorem~\ref{r>1}), obtained by substituting $\QQ$
for $V$.

By Theorem~\ref{bottstheorem}, the sheaves that appear in ${\bf
  F}_\bullet$ do not have any higher cohomology, so taking sections,
we get a $G$-equivariant acyclic complex
\begin{align} \label{sheafpieri} 0 \to \Hs^0(X; {\bf F}_n) \to \cdots
  \to \bigoplus_{i=1}^r \Hs^0(X; \BB \otimes \Sc_{\beta^i} \QQ) \to
  \Hs^0(X; \BB \otimes \Sc_\alpha \QQ) \to M(\alpha; \beta) \to 0,
\end{align}
where $M(\alpha; \beta)$ is by definition the cokernel of the map
preceding it. Letting $p \colon F \times X \to X$ denote the
projection onto the second factor, one has $p_*(\OO_\RR) = \BB$ by
\cite[Proposition 5.1.1(b)]{weyman}, so 
\[
\Hs^0(X; \BB \otimes \Sc_\lambda \QQ) \cong \Hs^0(F \times X; \OO_\RR
\otimes p^*(\Sc_\lambda \QQ))
\]
by the projection formula. Define $A = \Sym(F)$ to be the symmetric
algebra of $F$. Since $\OO_\RR$ is a quotient of $\OO_{F \times X}$,
the above groups inherit the structure of graded $A$-modules. If we
set the generator of $M(\alpha; \beta)$ to have degree 0, then the
module $\Hs^0(X; \BB \otimes \Sc_\lambda \QQ)$ in \eqref{sheafpieri}
needs to be shifted by degree $-|\lambda/\alpha|$ in order for the
differentials to be degree 0. However, these terms are not in general
free $A$-modules. So first we find free resolutions for each $\Hs^0(X;
\BB \otimes \Sc_\lambda \QQ)$ as $A$-modules. The maps in
\eqref{sheafpieri} will induce maps between these resolutions, and we
can put these resolutions together to get a free resolution of
$M(\alpha; \beta)$ as an $A$-module by taking an iterated mapping
cone.

Motivated by this, we call modules of the form $\Hs^0(X; \BB \otimes
\Sc_\lambda \QQ)$ {\bf geometric modules}, and let
$\mathbf{G}(\lambda)_\bullet$ denote their minimal free resolutions
over $A$. In order to calculate $\mathbf{G}(\lambda)_\bullet$, we will
need the following result.

\begin{theorem} \label{geometrictechnique} Let $Y$ be a projective
  variety over $K$ and $\mathcal{V}$ be a vector bundle over $Y$. Let
  $\EE = K^N \times Y$ denote the trivial vector bundle of rank $N$
  over $Y$, and let $\mathcal{S} \subset \EE$ be a subbundle with
  quotient bundle $\TT = \EE / \mathcal{S}$. Letting $A$ be the
  coordinate ring of $K^N$, there exists a complex
  $\mathbf{F}(\mathcal{V})_\bullet$ of free $A$-modules with minimal
  differentials of degree 0 whose terms are given by
  \[
  \mathbf{F}(\mathcal{V})_i = \bigoplus_{j \ge 0} \Hs^j(Y;
  \bigwedge^{i+j} (\TT^*) \otimes_{\OO_Y} \mathcal{V}) \otimes_K
  A(-i-j) \quad (i \in \Z),
  \]
  and whose homology groups are concentrated in degrees $i \le 0$,
  given by
  \[
  \Hs_i(\mathbf{F}(\mathcal{V})_\bullet) \cong \Hs^{-i}(Y;
  \Sym(\mathcal{S}^*) \otimes \mathcal{V}).
  \]
\end{theorem}

\begin{proof} See \cite[Theorem 5.1.2]{weyman}. \end{proof}

For our application, we take $Y = X$, $\mathcal{V} = \Sc_\lambda \QQ$,
$N = \dim F$, and $\mathcal{S} = \RR$. By Theorem~\ref{bottstheorem},
the complex $\mathbf{F}(\mathcal{V})_\bullet$ is exact in degrees $i
\ne 0$, so we get a free resolution of $\Hs^0(X; \BB \otimes
\Sc_\lambda \QQ)$ over $A$. In fact, this resolution is
$G$-equivariant (see \cite[Theorem 5.4.1]{weyman}). So we have reduced
the problem to calculating the cohomology groups $\Hs^j(X;
\bigwedge^{i+j} (\RR^\vee) \otimes \Sc_\lambda \QQ)$. We will do this
individually for the special orthogonal and symplectic groups in the
following subsections.

Now we can state the main result of this section.

\begin{theorem} The iterated mapping cone of the complexes
  $\mathbf{G}(\lambda)_\bullet$ resolving the terms in
  \eqref{sheafpieri} is a resolution of $M(\alpha; \beta)$. It is
  minimal in type {\rm C}. It is minimal in types {\rm B} and {\rm D}
  if there are no linear differentials in \eqref{sheafpieri}.
\end{theorem}

In order to prove this, we shall make use of the following technical
lemma.

\begin{lemma} \label{filtration} Suppose $\mathbf{F}_\bullet \to M \to
  0$ is a minimal acyclic complex of graded $A$-modules (i.e., the
  differentials have positive degree). Let $\mathbf{F}^i_\bullet \to
  \mathbf{F}_i$ be a minimal free resolution over $A$ with
  differentials denoted $\partial^i$ for each $i$, so that we have
  induced differentials $d \colon \mathbf{F}^i_\bullet \to
  \mathbf{F}^{i-1}_\bullet$. Suppose further that each
  $\mathbf{F}^i_\bullet$ has a filtration of subcomplexes $0 =
  \mathbf{F}^i_\bullet[-1] \subseteq \mathbf{F}^i_\bullet[0] \subseteq
  \mathbf{F}^i_\bullet[1] \subseteq \cdots \subseteq
  \mathbf{F}^i_\bullet$ such that:
  \begin{compactenum}[\rm (a)]
  \item \label{filtration1} Each homogeneous component of
    $\mathbf{F}_\bullet^i$ either intersects the $j$th graded part of
    the filtration in zero, or is entirely contained in it.
  \item \label{filtration2} For a homogeneous element $x$, let
    $\grade(x)$ be the number $g$ such that $x \in {\bf
      F}_\bullet^i[g] \setminus {\bf F}_\bullet^i[g-1]$. Then whenever
    $x \in \mathbf{F}^i_\bullet$ is homogeneous such that
    $\partial^i(x) \ne 0$, then $\deg(\partial^i(x)) - \deg(x) =
    \grade(x) - \grade(\partial^i(x)) + 1$.
  \item \label{filtration3} The induced maps $d \colon {\bf
      F}_\bullet^i \to {\bf F}_\bullet^{i-1}$ satisfy the inequality
    $\deg(d(x)) - \deg(x) \ge \grade(x) - \grade(d(x)) + 1$
    whenever $x$ is homogeneous, and that $\grade(D(x)) \le \grade(x)$
    for all differentials $D$ in the iterated mapping cone of the
    ${\bf F}^i_\bullet$.
  \end{compactenum}
  Then the iterated mapping cone of the $\mathbf{F}^i_\bullet$ forms a
  minimal resolution of $M$.
\end{lemma}

\begin{proof} The differentials in the mapping cone have the form $D
  \colon \mathbf{F}^i_k \to \mathbf{F}^{i-j}_{k-1+j}$. It is enough to
  prove that each such map is either 0, or has positive degree. More
  specifically, we will show by double induction on $j$ and $k$ that
  \begin{align} \label{finequality}
    \deg(D(x)) - \deg(x) \ge \grade(x) - \grade(D(x)) + 1
  \end{align}
  whenever $D(x) \ne 0$.

  The case $j = 0$ is the content of \eqref{filtration2} and the case
  $j = 1$ is the content of \eqref{filtration3}. Now suppose $j>1$. In
  general, if a nonzero differential $D \colon \mathbf{F}^i_k \to
  \mathbf{F}^{i-j}_{k-1+j}$ exists, then it was induced from a diagram
  \[
  \xymatrix{ \mathbf{F}^{i-j}_{k-2+j} & \ar[l]_-{d_1}
    \mathbf{F}^{i-\ell}_{k-1+\ell} &
    \ar[l]_-{d_2} \mathbf{F}^i_k \ar@{-->}[lld]^-D \\
    \mathbf{F}^{i-j}_{k-1+j} \ar[u]^-{\partial^{i-j}} }
  \]
  for some $\ell$ with $0 < \ell < j$. So each of $d_1$, $d_2$, and
  $\partial^{i-j}$ satisfies \eqref{finequality} by induction on
  $j$. Now we have for $x \in \mathbf{F}^i_k$
  \begin{align*}
    & (\deg(D(x)) - \deg(x)) + (\deg(\partial^{i-j}D(x)) - \deg(D(x)))\\
    =& (\deg(d_2(x)) - \deg(x)) + (\deg(d_1d_2(x)) - \deg(d_2(x)))\\
    \ge& (\grade(x) - \grade(d_2(x)) + 1) + (\grade(d_2(x)) -
    \grade(d_1d_2(x)) + 1).
  \end{align*}
  By (\ref{filtration2}), we have that $\deg(\partial^{i-j}D(x)) -
  \deg(D(x)) = \grade(D(x)) - \grade(\partial^{i-j}D(x)) + 1$, and
  using (\ref{filtration1}), we have $\grade(d_1d_2(x)) =
  \grade(\partial^{i-j}D(x))$. So we conclude that
  \[
  \deg(D(x)) - \deg(x) \ge \grade(x) - \grade(D(x)) + 1,
  \]
  as desired.
\end{proof}

To apply this lemma to our situation, we let
$\mathbf{G}(\lambda)_\bullet[0]$ be the subcomplex of
$\mathbf{G}(\lambda)_\bullet$ consisting of the $\Hs^0$ terms, and
$\mathbf{G}(\lambda)_\bullet[1] = \mathbf{G}(\lambda)_\bullet$. We
will show in each case that $\Hs^i$ terms are 0 for $i>1$. The fact
that the $\Hs^0$ terms form a subcomplex follows from minimality of
the complex in Theorem~\ref{geometrictechnique}. Also,
(\ref{filtration1}) and (\ref{filtration2}) follow from the grading
given by Theorem~\ref{geometrictechnique}, so in each case we will
only need to verify that (\ref{filtration3}) holds. For $x$
homogeneous coming from an $\Hs^1$ term, we have to show that
$\deg(d(x)) - \deg(x) \ge 2$ if $d(x)$ lies in an $\Hs^0$ term to
verify (\ref{filtration3}). Furthermore, if $x$ is homogeneous and
comes from an $\Hs^0$ term and $D$ is a differential in the iterated
mapping cone, we have to show that the components of $D(x)$ coming
from $\Hs^1$ terms are 0. In all other cases, we only need to show
that $\deg(d(x)) - \deg(x) \ge 1$ to verify (\ref{filtration3}).

\subsection{Type $\mathrm{B}_n$: Odd orthogonal
  groups.} \label{section:typeB} 

\begin{theorem} Let $G$ be of type $\mathrm{B}_n$. If $\lambda_n > 0$,
  then 
\begin{align} \label{typeB1}
  \Hs^0(X; \Sc_\lambda \QQ \otimes \bigwedge^i \RR^\vee) =
  \bigoplus_{\substack{\mu \subseteq \lambda \\ |\lambda/\mu| \in
      \{i,i-1\} \\ (\lambda, \mu) \in \HS }} V_\mu,
\end{align}
and all higher cohomology vanishes. If $\lambda_n = 0$, then
\begin{align}
  \Hs^0(X; \Sc_\lambda \QQ \otimes \bigwedge^i \RR^\vee) &=
  \bigoplus_{\substack{ \mu \subseteq \lambda \\ |\lambda / \mu| = i \\
      (\lambda, \mu) \in \HS}} V_\mu \label{typeB2} \\
  \Hs^1(X; \Sc_\lambda \QQ \otimes \bigwedge^i \RR^\vee) &=
  \bigoplus_{\substack{ \mu \subseteq \lambda \\ |\lambda / \mu| = i-2
      \\ (\lambda, \mu) \in \HS} } V_\mu, \quad (\text{if } i \ge
  2), \label{typeB3}
\end{align}
and all other cohomology vanishes.
\end{theorem}

\begin{proof}
  First we calculate the cohomology groups of $\Sc_\lambda \QQ \otimes
  \bigwedge^i \RR$. We use that $\RR = \QQ^*$ and hence $\bigwedge^i
  \RR = \bigwedge^{n-i} \QQ \otimes (\bigwedge^n \QQ)^{-1}$. So by
  Pieri's formula, $\Sc_\lambda \QQ \otimes \bigwedge^i \RR =
  \bigoplus_\mu \Sc_\mu \QQ \otimes (\bigwedge^n \QQ)^{-1}$ where the
  sum is over all $\mu$ such that $|\mu/\lambda| = n-i$ and $\mu_j -
  \lambda_j \le 1$ for all $j$. The highest weight of the
  representation $\Sc_\mu \QQ \otimes (\bigwedge^n \QQ)^{-1}$ is
  $(\mu_1 - 1, \dots, \mu_n - 1)$. Since this is a dominant weight for
  $G$ if and only if $\mu_n - 1 \ge 0$, we conclude from
  Theorem~\ref{bottstheorem} that
  \begin{align} \label{typeBH0} \Hs^0(X; \Sc_\lambda \QQ \otimes
    \bigwedge^i \RR) =
    \bigoplus_{\substack{ \mu \subseteq \lambda \\ |\lambda / \mu| = i \\
        (\lambda, \mu) \in \HS}} V_\mu.
  \end{align}
  Now suppose that $\mu_n = 0$ (which can only happen if $\lambda_n =
  0$). Let $w \in W$ be the reflection given by the simple root
  $\eps_n$, i.e., it acts on weights by changing the sign of the last
  coordinate. In this case, $\rho = (\frac{2n-1}{2}, \frac{2n-3}{2},
  \dots, \frac{1}{2})$. So $w^\bullet(\mu_1 - 1, \dots, \mu_{n-1} - 1,
  -1) = (\mu_1 - 1, \dots, \mu_{n-1} - 1, 0)$. If $\mu_{n-1} \ge 1$,
  this weight is dominant. Otherwise, if $\mu_{n-1} = 0$, let $w_1 \in
  W$ be the reflection given by the simple root $\eps_{n-1} - \eps_n$,
  which permutes the last two coordinates. Then $w_1^\bullet$ fixes
  $(\mu_1 - 1, \dots, -1, 0)$, so there is no cohomology in this
  case. Hence we conclude that
  \begin{align} \label{typeBH1} \Hs^1(X; \Sc_\lambda \QQ \otimes
    \bigwedge^i \RR) =
    \bigoplus_{\substack{ \mu \subseteq \lambda \\
        |\lambda / \mu| = i-1 \\ (\lambda, \mu) \in \HS } } V_\mu,
    \quad (\text{if } \lambda_n = 0 \text{ and } i \ge 1).
  \end{align}

  There is a nonsplit exact sequence
  \[
  0 \to \RR \to \RR^\vee \to \OO_X \to 0,
  \]
  where the trivialization of the cokernel comes from the fact that it
  has rank 1 and the corresponding $P$-module has the zero
  weight. This sequence gives rise to a short exact sequence
  \[
  0 \to \bigwedge^i \RR \to \bigwedge^i \RR^\vee \to \bigwedge^{i-1}
  \RR \to 0.
  \]
  Now we tensor with $\Sc_\lambda \QQ$ and take the long exact
  sequence of cohomology to get
  \begin{align*}
    0 \to & \Hs^0(X; \Sc_\lambda \QQ \otimes \bigwedge^i \RR) \to
    \Hs^0(X; \Sc_\lambda \QQ \otimes \bigwedge^i \RR^\vee) \to
    \Hs^0(X; \Sc_\lambda \QQ \otimes \bigwedge^{i-1} \RR)
    \xrightarrow{\delta} \\
    & \Hs^1(X; \Sc_\lambda \QQ \otimes \bigwedge^i \RR) \to \Hs^1(X;
    \Sc_\lambda \QQ \otimes \bigwedge^i \RR^\vee) \to \Hs^1(X;
    \Sc_\lambda \QQ \otimes \bigwedge^{i-1} \RR) \to 0.
  \end{align*}

  By \eqref{typeBH1}, the higher cohomology vanishes if $\lambda_n \ge
  1$, and we conclude \eqref{typeB1} by semisimplicity of $G$. For
  $\lambda_n = 0$, \eqref{typeB2} and \eqref{typeB3} follow from
  \eqref{typeBH0} and \eqref{typeBH1} if we can show that $\delta$ is
  an isomorphism, and this is the content of Lemma~\ref{deltaiso}.
\end{proof}

\begin{lemma} \label{deltaiso} If $\lambda_n = 0$ and $i \ge 1$, then
  $\delta$ is an isomorphism.
\end{lemma}

\begin{proof}
  Consider the short exact sequence
  \begin{align} \label{sheafsequence}
  0 \to \Sc_\lambda \QQ \otimes \bigwedge^i \RR \to \Sc_\lambda \QQ
  \otimes \bigwedge^i \RR^\vee \to \Sc_\lambda \QQ \otimes
  \bigwedge^{i-1} \RR \to 0.
  \end{align}
  Let $\mu \subseteq \lambda$ be a partition such that $(\lambda, \mu)
  \in \HS$ and $|\lambda / \mu| = i-1$, so that $V_\mu$ is a
  subrepresentation of $\Hs^0(X; \Sc_\lambda \QQ \otimes
  \bigwedge^{i-1} \RR)$ and of $\Hs^1(X; \Sc_\lambda \QQ \otimes
  \bigwedge^i \RR)$. Inside of $\Sc_\lambda \QQ \otimes \bigwedge^i
  \RR$ is a subbundle isomorphic to $\Sc_{\mu - \eps_n} \QQ$ (since
  $\lambda_n = 0$), and inside of $\Sc_\lambda \QQ \otimes
  \bigwedge^{i-1} \RR$ is a subbundle isomorphic to $\Sc_\mu \QQ$. Let
  $\EE$ be the subbundle of $\Sc_\lambda \QQ \otimes \bigwedge^i
  \RR^\vee$ which is the extension of these two subbundles:
  \[
  0 \to \Sc_{\mu - \eps_n} \QQ \to \EE \to \Sc_\mu \QQ \to 0.
  \]
  In order to show that $\delta$ is an isomorphism between the two
  copies of $V_\mu$, it is enough to show that $\EE$ is a nontrivial
  extension. To show this, we will use the equivalence between
  homogeneous bundles and rational $P$-modules, and show that the
  corresponding short exact sequence on the fiber over the $P$-fixed
  point of $X$ is a non-split extension.

  Let $R$, $R^\vee$, and $Q$ denote the fibers of $\RR$, $\RR^\vee$,
  and $\QQ$, respectively, over the $P$-fixed point of $X =
  G/P$. Recall that we have a basis $e_1, \dots, e_{2n+1}$ for $F$
  such that $e_i$ and $e_{2n+2-i}$ are dual basis vectors. Then
  \[
  R = \langle e_1, \dots, e_n \rangle,\quad R^\vee = \langle e_1,
  \dots, e_n, e_{n+1} \rangle,\quad Q = \langle e^*_1, \dots, e^*_n
  \rangle = \langle e_{2n+1}, \dots, e_{n+2} \rangle.
  \]
  Also, let $G_0$ be the subgroup of $P$ which leaves $\langle
  e_{n+1}, \dots, e_{2n+1} \rangle$ invariant. By definition, $P$
  leaves $R$ invariant, so $G_0 \cong \GL(R)$, and hence is linearly
  reductive. Thus any rational $P$-module is completely reducible as a
  $G_0$-module, and we will refer to such a decomposition as the
  associated graded module. To keep track of the $P$-action, we should
  write the fiber of \eqref{sheafsequence} as
  \[
  0 \to \Sc_\lambda Q \otimes \bigwedge^i R \to \Sc_\lambda Q \otimes
  \bigwedge^i R^\vee \to \Sc_\lambda Q \otimes \bigwedge^{i-1} R
  \otimes R^\vee / R \to 0.
  \]

  Let $g \in P$ be defined by $g(e_{n+1}) = e_n + e_{n+1}$,
  $g(e_{n+2}) = -\frac{1}{2}e_n - e_{n+1} + e_{n+2}$ and $g(e_i) =
  e_i$ for $i \notin \{ n+1, n+2\}$. (We are only interested in the
  value of $g(e_{n+1})$, the value of $g(e_{n+2})$ given is needed to
  ensure that $g$ preserves the symplectic form.) Then $\Sc_\mu Q$ is
  a summand of $\Sc_\lambda Q \otimes \bigwedge^{i-1} R \otimes R^\vee
  / R$, and hence is a summand of the associated graded of
  $\Sc_\lambda Q \otimes \bigwedge^i R^\vee$. We claim that $g$ does
  not fix $\Sc_\mu Q$, which will show that $\EE$ is a nontrivial
  extension of $\Sc_\mu \QQ$ and $\Sc_{\mu - \eps_n} \QQ$.

  Consider an Olver map (see Corollary~\ref{exteriorolver})
  \[
  \Sc_\mu Q \otimes \bigwedge^n Q \to \Sc_\lambda Q \otimes
  \bigwedge^{n-i+1} Q.
  \]
  Since $\mu_n = 0$, the canonical tableau $T$ inside of $\Sc_\mu Q
  \otimes \bigwedge^n Q$ contains a single $e_n^*$ in the $n$th
  column, and its image is a linear combination over possible ways to
  remove $n-i+1$ boxes from $T$. Since $\lambda_n = 0$, some of these
  summands will be of the form $T_\lambda \otimes e^*_I \wedge e^*_n$
  where $e^*_I \in \bigwedge^{n-i} Q$. Hence when we tensor this Olver
  map with $\bigwedge^n R \otimes R^\vee / R$, we can contract this
  $e^*_n$ with a copy of $e_n$ for such summands. The main point is
  that under the isomorphism
  \[
  \Sc_\lambda Q \otimes \bigwedge^{n-i+1} Q \otimes \bigwedge^n R
  \otimes R^\vee / R \cong \Sc_\lambda Q \otimes \bigwedge^{i-1} R
  \otimes R^\vee / R,
  \] 
  the image of the canonical tableau $T$ is a linear combination $T =
  \sum_j c_j T_j \otimes e_{I(j)} \otimes e_{n+1}$ where for at least
  one value of $j$ with $c_j \ne 0$, we have $e_{I(j)} \wedge e_{n+1}
  \ne 0$. But the action of $g$ on $T$ is given by replacing the
  $e_{n+1}$ factor by $e_n + e_{n+1}$, so $gT$ does not lie in
  $\Sc_\mu Q$, but in the span of both $\Sc_\mu Q$ and $\Sc_{\mu -
    \eps_n} Q$. So we have established that $\EE$ is a nontrivial
  extension.

  Thus, $\EE$ cannot have a $V_\mu$-isotypic component in its global
  sections. This follows, for example, from the identification of
  cohomology of homogeneous bundles with an induction functor (see
  \cite[Proposition I.5.12]{jantzen}) and Frobenius reciprocity (see
  \cite[Proposition I.3.4]{jantzen}) since the extension of $\Sc_\mu Q$
  and $\Sc_{\mu-\eps_n} Q$ has no $P$-submodule isomorphic to $\Sc_\mu
  Q$.
\end{proof}

We now show that the hypothesis of
Lemma~\ref{filtration}(\ref{filtration3}) is satisfied. Suppose that
$\lambda(i)$ and $\lambda(i+1)$ are partitions appearing in the $i$th
and $(i+1)$st degrees of the Pieri resolution ${\bf F}(\alpha;
\beta)_\bullet$ such that the differential
\[
\Hs^0(X; \Sc_{\lambda(i+1)} \QQ \otimes \bigwedge^j \RR^\vee) \to
\Hs^0(X; \Sc_{\lambda(i)} \QQ \otimes \bigwedge^j \RR^\vee)
\]
is not minimal. The first term has homogeneous degree $-|\lambda(i+1)
/ \alpha| - j$ and the second has homogeneous degree $-|\lambda(i) /
\alpha| - j$, so in order for this happen, we would need that
$|\lambda(i)| = |\lambda(i+1)|$. But in this case, this differential
is induced by a degree 0 map of the form $\Sc_{\lambda(i+1)} \QQ
\otimes \BB \to \Sc_{\lambda(i)} \QQ \otimes \BB$ which is zero unless
$\lambda(i+1) = \lambda(i)$. But in the latter case, the Pieri
resolution we started with is not minimal. 

If we have a horizontal map from an $\Hs^1$ term to an $\Hs^0$ term
(and hence $\lambda(i)_n = \lambda(i+1)_n = 0$), then it must have
degree at least 2: if it had smaller degree, then $|\lambda(i+1)| -
|\lambda(i)| \le 1$, but then \eqref{typeB2} and \eqref{typeB3} show
that the map must be 0. Hence the hypothesis of
Lemma~\ref{filtration}(\ref{filtration3}) holds.

Finally, we have to show that any higher differential from an $\Hs^0$
term to an $\Hs^1$ term in the iterated mapping cone of
\eqref{sheafpieri} is 0. Here we will use the assumption that the
differentials in \eqref{sheafpieri} have degree $>1$. We consider a
map of the form
\[
\Hs^0(X; \Sc_{\lambda(i+k)} \QQ \otimes \bigwedge^j \RR^\vee) \to
\Hs^1(X; \Sc_{\lambda(i)} \QQ \otimes \bigwedge^{j+k} \RR^\vee),
\]
where $k \ge 1$. This map is equivariant, so there is some $\mu$ for
which $V_\mu$ appears in both modules, so pick such a $\mu$. Then by
\eqref{typeB3}, we have $|\mu| = |\lambda(i)| - j-k+2$. Also, $|\mu| =
|\lambda(i+k)| - j$ or $|\mu| = |\lambda(i+k)| - j + 1$ by
\eqref{typeB1} and \eqref{typeB2}. This implies that $|\lambda(i+k)| -
|\lambda(i)| \le 2-k$. The higher differential lifts the map $A
\otimes V_{\lambda(i+k)} \to A \otimes V_{\lambda(i)}$ by
\eqref{typeB1}. But such a map has degree $\le 2-k \le 1$, so is 0
because we assumed that \eqref{sheafpieri} does not have linear
differentials.

~

More generally, suppose that $N$ is the cokernel of an equivariant map
of the form 
\[
\bigoplus_{i=1}^r A(-|\beta^i / \alpha|) \otimes V_{\beta^i} \to A
\otimes V_{\alpha}
\]
where each of $\beta^1, \dots, \beta^r$ is obtained by adding vertical
strips to $\alpha$. Let $M$ be the sections of the cokernel of the map
of sheaves $\bigoplus_{i=1}^r \Sc_{\beta^i} \QQ \otimes \BB \to
\Sc_{\alpha} \QQ \otimes \BB$. From the description above, we have a
surjection $N \to M \to 0$. Letting $N_0$ be the kernel, we know that
$N_0$ is generated by all partitions obtained by removing a box from
$\alpha$. If there was only one such partition, we can resolve $N_0$
by induction on the size of $\alpha$. Otherwise, call these partitions
$\alpha'_1, \dots, \alpha'_s$, and add the relations $\alpha'_{i+1},
\dots, \alpha'_s$ to $N$ to get a new module $L_i$. Then we have a
short exact sequence
\[
0 \to L'_i \to L_i \to L_{i-1} \to 0
\]
where $L'_i$ is generated by $\alpha'_i$. So we can get a resolution
for $L_i$ via resolutions of both $L'_i$ and $L_{i-1}$. Note that $L_0
= M$ and $L_s = N$, so we have described a sequence of quotients 
\[
N = L_s \to L_{s-1} \to \cdots \to L_1 \to L_0 = M
\]
for obtaining a resolution of $N$. Unfortunately, it will usually be
far from minimal, and one must know something about which
cancellations will occur. These remarks also apply to the cases of the
symplectic group and the even orthogonal group.

\subsection{Type $\mathrm{C}_n$: Symplectic groups.}

\begin{theorem} Let $G$ be a group of type $\mathrm{C}_n$. Then
\begin{align} \label{typeC1}
\Hs^0(X; \Sc_\lambda \QQ \otimes \bigwedge^i \RR^\vee) =
\bigoplus_{\substack{ \mu \subseteq \lambda \\ |\lambda / \mu| = i \\
    (\lambda, \mu) \in \HS}} V_\mu,
\end{align}
and all higher cohomology vanishes.
\end{theorem}

\begin{proof} The dominant weights are the same for the symplectic and
  odd orthogonal groups, so using \eqref{typeBH0} and the fact that
  $\RR^\vee = \RR$ gives \eqref{typeC1}.

  Furthermore, if $\mu_n = 0$, then $\Sc_{\mu} \QQ \otimes
  (\bigwedge^n \QQ)^{-1}$ has no cohomology. To see this, let $w \in
  W$ be the reflection given by the simple root $2\eps_n$, which
  changes the sign of the last coordinate. Then $w^\bullet$ fixes
  $(\mu_1 - 1, \dots, \mu_n - 1)$ since $\rho = (n, n - 1, \dots, 2,
  1)$.
\end{proof}

The application of Lemma~\ref{filtration} follows as in the previous
section.

\subsection{Type $\mathrm{D}_n$: Even orthogonal groups.}

\begin{theorem} Let $G$ be a group of type $\mathrm{D}_n$. Then
\begin{align} \label{typeD1} \Hs^0(X; \Sc_\lambda \QQ \otimes
  \bigwedge^i \RR^\vee) =
  \bigoplus_{\substack{ \mu \supseteq \lambda \\
      \mu_{n-1} - 1 \ge |\mu_n - 1| \\ |\mu / \lambda| = n-i \\
      (\mu, \lambda) \in \HS}} V^*_{(\mu_1 - 1, \dots, \mu_n - 1)}.
\end{align}
If $\lambda_{n-1} = \lambda_n = 0$ and $i \ge 2$, we also have
\begin{align} \label{typeD2} 
  \Hs^1(X; \Sc_\lambda \QQ \otimes \bigwedge^i \RR^\vee) =
  \bigoplus_{\substack{ \mu \subseteq \lambda \\ |\lambda / \mu | =
      i-2 \\ (\lambda, \mu) \in \HS }} V_\mu, 
\end{align}
and all other cohomology vanishes.
\end{theorem}

Note that by Remark~\ref{dualweights}, we only need to refer to the
dual in \eqref{typeD1}.

\begin{proof} As before, $\Sc_\lambda \QQ \otimes \bigwedge^i \RR^\vee
  = \bigoplus_\mu \Sc_\mu \QQ \otimes (\bigwedge^n \QQ)^{-1}$ where
  the sum is over all $\mu$ such that $|\mu/\lambda| = n-i$ and $\mu_j
  - \lambda_j \le 1$ for all $j$. The highest weight of the
  representation $\Sc_\mu \QQ \otimes (\bigwedge^n \QQ)^{-1}$ is
  $(\mu_1 - 1, \dots, \mu_n - 1)$. This is a dominant weight if and
  only if $\mu_{n-1} - 1 \ge |\mu_n - 1|$, so we conclude
  \eqref{typeD1} from Theorem~\ref{bottstheorem}. Note that the
  condition $\mu_{n-1} - 1 \ge |\mu_n - 1|$ happens in exactly two
  cases: if $\mu_n \ge 1$, or if $\mu_n = 0$ and $\mu_{n-1} \ge 2$. So
  we need to study the cases when $\mu_n = 0$ and $\mu_{n-1} \in
  \{0,1\}$. Of course, this forces $\lambda_n = 0$.

  Let $w \in W$ be the reflection given by the simple root $\eps_{n-1}
  + \eps_n$. This acts on weights by
  \[
  w(\alpha_1, \dots, \alpha_{n-1}, \alpha_n) = (\alpha_1, \dots,
  \alpha_{n-2}, -\alpha_n, -\alpha_{n-1}).
  \]
  Also, we have $\rho = (n-1, n-2, \dots, 1, 0)$. In the case that
  $\mu_{n-1} = 1$, we see that $w^\bullet(\mu_1 - 1, \dots, 0, -1) =
  (\mu_1 - 1, \dots, 0, -1)$, so there is no cohomology. On the other
  hand, if $\mu_{n-1} = 0$ (which can only happen if $i \ge 2$ and
  $\lambda_{n-1} = 0$), then we have $w^\bullet(\mu_1 - 1, \dots,
  \mu_{n-2} - 1, -1, -1) = (\mu_1 - 1, \dots, \mu_{n-2} - 1, 0,
  0)$. There are two possibilities: either $\mu_{n-2} \ge 1$ or
  $\mu_{n-2} = 0$. In the first case, the resulting weight is
  dominant. In the second case, let $w_1 \in W$ be the reflection
  given by the simple root $\eps_{n-2} - \eps_{n-1}$, which permutes
  the $(n-2)$nd and $(n-1)$st coordinates. Then $w_1^\bullet(\mu_1 -
  1, \dots, -1, 0, 0) = (\mu_1 - 1, \dots, -1, 0, 0)$, so there is no
  cohomology. We conclude \eqref{typeD2} and all higher cohomology of
  $\Sc_\lambda \QQ \otimes \bigwedge^i \RR^\vee$ vanishes.
\end{proof}

Now we apply Lemma~\ref{filtration}. There are only two possibilities
for a horizontal differential to be non-minimal: the first is when
both terms are $\Hs^0$ terms or both are $\Hs^1$ terms (and these are
ruled out just as in the odd orthogonal case), and the second is when
one term is an $\Hs^0$ term and the other is an $\Hs^1$. First suppose
the horizontal differential maps an $\Hs^1$ term to an $\Hs^0$ term.
In this case, the differential has degree $\le 1$ only if the
partitions from \eqref{sheafpieri} that they resolve have the same
size, so the corresponding map must be 0. 

Now consider a higher differential in the iterated mapping cone of
\eqref{sheafpieri} of the form
\[
\Hs^0(X; \Sc_{\lambda(i+k)} \QQ \otimes \bigwedge^j \RR^\vee) \to
\Hs^1(X; \Sc_{\lambda(i)} \QQ \otimes \bigwedge^{j+k} \RR^\vee),
\]
for $k \ge 1$. We want to show that this map is identically 0. Suppose
that $V_\mu$ is a representation that appears in both modules. By
\eqref{typeD2}, $\mu$ is a partition with $\mu_n = 0$ such that $|\mu|
= |\lambda(i)| - j-k+2$. By \eqref{typeD1}, we get $|\mu| =
|\lambda(i+k)| - j$, so $|\lambda(i)| + 2-k = |\lambda(i+k)|$. So the
higher differential is 0 by similar considerations as in
Section~\ref{section:typeB}. 

\section{Toward equivariant Boij--S\"oderberg
  cones.} \label{equivariantbssection}

To give some context, we review the non-equivariant version of
Boij--S\"oderberg cones in Section~\ref{bsconesection} and present the
Boij--S\"oderberg algorithm for writing Betti tables as linear
combinations of pure Betti tables. In Section~\ref{conjecturesection},
we formulate a conjectural equivariant analogue of Boij--S\"oderberg
decompositions and provide some partial results. Finally, in
Section~\ref{equivariantexamples}, we present some examples of these
decompositions, and show that the equivariant analogue of the
Boij--S\"oderberg algorithm does not hold.

\subsection{Boij--S\"oderberg cones in general.} \label{bsconesection}

Let $A = K[x_1, \dots, x_n]$ as usual, and pick $c \le n$. The {\bf
  Boij--S\"oderberg cone}, denoted $\Delta$, is the cone generated by
all Betti tables corresponding to Cohen--Macaulay modules of
codimension $c$ with pure free resolutions. Given two degree sequences
$d = (d_0, \dots, d_c)$ and $d' = (d'_0, \dots, d'_c)$, we say that $d
\le d'$ if $d_i \le d'_i$ for $i=0,\dots,n$. Let $\Pi$ denote the
poset of all degree sequences. Any maximal chain $C$ in $\Pi$ forms a
simplicial cone inside of $\Delta$ by taking the cone generated by the
pure Betti tables coming from the degree sequences of $C$. By
\eqref{herzogkuhl}, these simplicial cones are well-defined. In fact,
the union of all such $C$ forms a simplicial fan $\mathcal{F}$ (see
\cite[Proposition 2.9]{bsconj}) which we call the {\bf
  Boij--S\"oderberg fan}. Recall from Section~\ref{purefreesection}
that $\B(M)$ denotes the graded Betti table of a module $M$. The
following was conjectured by Boij and S\"oderberg:

\begin{theorem}[Eisenbud--Schreyer] \label{bscone} Let $K$ be a field
  of arbitrary characteristic. Let $M$ be a finitely generated
  Cohen--Macaulay graded $A$-module of codimension $c$. Then
  $\B(M)$ can be written as a positive rational linear combination
  of Betti tables of pure Cohen--Macaulay modules of codimension
  $c$. This linear combination is unique if we require that the degree
  sequences of these pure free diagrams form a chain in $\Pi$.
\end{theorem}

\begin{proof} See \cite[\S 7]{es}. \end{proof}

Hence the cone $\Delta$ contains the Betti tables of all finitely
generated Cohen--Macaulay graded $A$-modules of codimension $c$. There
is a simple algorithm for producing the linear combination whose idea
originally appeared in \cite[\S 2.3]{bsconj}. First, define the {\bf
  impurity} $i(\beta)$ of a Betti diagram $\beta$ to be the number of
its nonzero entries minus the number of nonzero columns. So a pure
diagram has impurity 0. Given a Betti diagram $\B$, define $d_i =
\min\{j \mid \B_{i,j} \ne 0\}$ for $i=0,\dots,c$. This is the {\bf top
  degree sequence} of $\B$. The {\bf bottom degree sequence} can be
defined by replacing min with max. Let $D$ be a pure Betti diagram of
degree sequence $d = (d_0, \dots, d_c)$: the Herzog--K\"uhl equations
\cite[Theorem 1]{hk} state that if $\B(M)$ is a pure Betti table of
degree $d$, then one has
\begin{align} \label{herzogkuhl} 
  \B(M)_{i,d_i} = (-1)^{i+1} q \prod_{j \not\in \{0, i\}} \frac{d_j -
    d_0}{d_j - d_i}
\end{align}
for some positive rational number $q$ if and only if $M$ is
Cohen--Macaulay. Pick this rational multiple to be minimal with
respect to the property of $D$ having integral entries (this might not
be the Betti diagram of a module), and let $r$ be the largest rational
number such that $\B' = \B - rD$ has non-negative entries. Then from
Theorem~\ref{bscone}, $i(\B') < i(\B)$, so we repeat the process,
which must terminate after finitely many steps.

Note that our choice of $D$ at each stage ensures that the degree
sequences of the pure diagrams used in the resulting linear
combination which expresses $\B$ will form a chain in
$\Pi$. Example~\ref{countersimplicial} will show that the equivariant
analogue of this algorithm does not hold.

\subsection{An equivariant conjecture.} \label{conjecturesection}

In this section, we study the Betti tables of equivariant
Cohen--Macaulay modules. The support of an equivariant module is also
equivariant, hence must either be all of $\Spec A$, or the homogeneous
maximal ideal $(x_1, \dots, x_n)$. We conclude that a Cohen--Macaulay
equivariant $A$-module is either free or has finite length.

Given an equivariant graded free resolution of a finite length
(codimension $n$) Cohen--Macaulay $A$-module $M$, we wish to
categorify the expression of $\B(M)$ as a linear combination of pure
free diagrams. Without loss of generality, we will henceforth assume
that $M$ is a polynomial representation. In particular, by clearing
denominators and taking high enough multiples, we can replace the pure
free diagrams by the equivariant pure free diagrams of
Section~\ref{purefreesection}, and the integer coefficients $r$ will
be replaced by Schur positive symmetric functions (the characters of
some $\GL(V)$-representation). To get an equivariant Betti diagram out
of an equivariant module $M$, we let $\B(M)_{i,j}$ be the character of
the minimal generating representations in degree $j$ of the $i$th
syzygy module of an equivariant graded minimal free resolution of
$M$, as described in the introduction.

The equivariant Betti diagram follows the usual convention of Betti
diagrams, namely that the $i$th column and $j$th row contains
$\B(M)_{i,j-i}$. We will say that $\B(M)$ is {\bf pure} if each column
contains at most one nonzero entry, just as in the non-equivariant
setting.

In order to make things precise, first define $\SQ = \SQ(n)$ to be the
quotient field of the ring of symmetric functions $\Z[x_1, \dots,
x_n]^{\mathfrak{S}_n}$. Every element of $\SQ$ can be written as $A/B$
where $A$ and $B$ are symmetric functions. Supposing that one can
write $A = \sum_\lambda a_\lambda s_\lambda$ and $B = \sum_\lambda
b_\lambda s_\lambda$ such that $a_\lambda \ge 0$ and $b_\lambda \ge 0$
for all $\lambda$, we say that $A/B$ is {\bf Schur positive}, and also
write $A/B \ge_s 0$ and $A/B \in \SQ_{\ge 0}$. Note that the product
of two Schur positive fractions is still Schur positive. This notion
of Schur positive fractions seems to be the correct replacement for
positive rational numbers in the equivariant setting.

The equivariant graded Betti tables live in the $\SQ$-vector space
$\SB = \bigoplus_{-\infty}^\infty \SQ^{n+1}$, where we think of the
elements of this vector space as tables with $n+1$ columns and
infinitely many rows. If $\underline{d} \le \ol{d}$ are two degree
sequences, let $\SB_{\underline{d}, \ol{d}}$ be the finite dimensional
subspace of $\SB$ consisting of those tables whose nonzero entries lie
within $[\underline{d}, \ol{d}]$. We are interested in the ``cone''
$C$ whose generators are the pure equivariant resolutions of
Section~\ref{purefreesection}. By a cone with generators, we mean the
set of finite linear combinations of the generators using Schur
positive coefficients. While $C$ is not a cone in the usual sense, $C$
is convex when we consider $\SB$ as a $\Q$-vector space. Given a set
of elements $S \subseteq \SB$, we write $\SQ_{\ge 0}S$ to denote the
set of finite Schur positive linear combinations of elements of $S$.

\begin{problem}[Weak version] \label{weakconj} Let $K$ be a field of
  characteristic 0. Let $M$ be an equivariant Cohen--Macaulay graded
  $A$-module of finite length. Is it true that $\B(M)$ can be written
  as a Schur positive linear combination of Betti tables of pure
  Cohen--Macaulay modules of finite length?
\end{problem}

It is not too hard to answer Problem~\ref{weakconj} affirmatively in
the case when the impurity is concentrated in one column (see
Proposition~\ref{simplicialcase}). If Problem~\ref{weakconj} can be
answered affirmatively in the general case, then we can further ask if
the equalities hold on the level of complexes.

\begin{problem}[Strong version] \label{strongconj} Let $M$ be a finite
  length equivariant $A$-module with equivariant minimal free
  resolution ${\bf F}_\bullet$. Does there exist degree sequences
  $d^1, \dots, d^r$ and representations $W, W_1, \dots, W_r$ such that
  $W \otimes {\bf F}_\bullet$ has a filtration of subcomplexes whose
  associated graded is isomorphic to
  \[
  \bigoplus_{i=1}^r W_i \otimes {\bf F}(d^i)_\bullet,
  \]
  where the ${\bf F}(d^i)_\bullet$ is the pure resolution of degree
  $d^i$ described in Section~\ref{purefreesection}, and the
  isomorphism is of equivariant complexes? Can we write $W \otimes
  {\bf F}_\bullet$ as a direct sum of subcomplexes instead of having
  to pass to an associated graded?
\end{problem}

Problem~\ref{strongconj} is false in the non-equivariant case as the
next example shows.

\begin{eg} Let $A = K[x,y]$ and let $M = A / (x,y^2)$. Then $M$ is
  Cohen--Macaulay and has Betti table $\left( \begin{array}{ccc} 1 & 1
      & - \\ - & 1 & 1 \end{array} \right)$. For any positive integer
  $n$, we cannot find a filtration $0 \to N \to M^{\oplus n} \to
  M^{\oplus n} / N \to 0$ of $M^{\oplus n}$ such that each piece has a
  pure free resolution because $x$ would annihilate both $N$ and
  $M^{\oplus n} / N$, which means that the middle entry of the first
  row of their Betti tables must be nonzero.
\end{eg}

We indicate some facts which may be of use in trying to answer
Problem~\ref{weakconj} affirmatively. However, we first point out a
fact which makes finding a counterexample particularly difficult.

\begin{proposition} \label{monomialpositive} If $A$ is any weight
  positive (positive in the monomial symmetric function basis)
  symmetric function, then there exists a Schur polynomial $s_\lambda$
  such that $As_\lambda$ is Schur positive.
\end{proposition}

\begin{proof} Given two Schur polynomials $s_\lambda$ and $s_\mu$, let
  $W_1, \dots, W_N$ be the weights of $\mu$. Then for each $\lambda +
  W_i$, there either exists a nonidentity $\sigma \in \SS_n$ such that
  $\sigma^\bullet(\lambda + W_i) = \lambda + W_i$, or there exists a
  unique $\sigma_i$ such that $\sigma^\bullet(\lambda + W_i) =
  \lambda^i$ is a dominant weight (a partition). Recall that
  $\sigma^\bullet(\lambda + W_i)$ is defined to be $\sigma(\lambda +
  W_i + \rho) - \rho$ where $\rho = (n-1, n-2, \dots, 1, 0)$. In the
  second case, we say that $\lambda + W_i$ is nondegenerate, and we
  claim that $s_\lambda s_\mu = \sum_i (-1)^{\ell(\sigma_i)}
  s_{\lambda^i}$, the sum over $i$ such that $\lambda + W_i$ is
  nondegenerate. First, define $a_\gamma = \det(x_j^{\gamma_i + n -
    i})_{i,j=1}^n$ for all $\gamma \in {\bf N}^n$. For a weight $W$,
  set $o_W = m_W^{-1} \sum_{\tau \in \SS_n^W} x^{\tau(W)}$ where
  $x^{W_i} = x_1^{W_i(1)} \cdots x_n^{W_i(n)}$ and $m_W$ is the
  integer needed so that the coefficients in $o_W$ are 1. Hence we
  have $s_\mu = \sum_W c_W o_W$ for some coefficients $c_W$. The equation
  \begin{align*}
    s_\lambda o_W &= \frac{a_{\lambda + \rho} o_W}{a_\rho} &
    (\text{by } \eqref{weylcharacterformula}) \\
    &= a_\rho^{-1} m_W^{-1} \sum_{\sigma \in \SS_n}
    (-1)^{\ell(\sigma)} x^{\sigma(\lambda + \rho)} \sum_{\tau \in
      \SS_n} x^{\tau(W)}\\ 
    &= a_\rho^{-1} m_W^{-1} \sum_{\sigma, \tau' \in \SS_n}
    (-1)^{\ell(\sigma)} x^{\sigma(\lambda + \rho + \tau'(W))} &
    (\text{setting } \tau' = \sigma^{-1}\tau)\\
    &= a_\rho^{-1} m_W^{-1} \sum_{\tau \in \SS_n}
    (-1)^{\ell(\sigma_\tau)} a_{\sigma_\tau(\lambda + \rho + \tau(W))}
  \end{align*}
  is valid for arbitrary choices of $\sigma_\tau \in \SS_n$. We have
  that $\lambda + \tau(W)$ is degenerate (say corresponding to the
  permutation $\sigma_\tau$) if and only if the determinant
  $a_{\sigma_\tau(\lambda + \rho + \tau(W))}$ is 0 since this
  corresponds to the matrix having repeated rows. Hence only
  nondegenerate weights contribute to the sum $s_\lambda \sum_W c_W
  o_W$, and in the nondegenerate case, one has $a_{\sigma_\tau(\lambda
    + \rho + \tau(W))} / a_\rho = s_{\sigma^\bullet_\tau(\lambda +
    \tau(W))}$ assuming that $\sigma_\tau$ has been chosen so that the
  subscript of $s$ is a partition. This proves the claim.

  In our case, we can choose $\lambda$ such that for every weight $W$
  appearing in $A$, we have that $\lambda + W$ is dominant, and hence
  if $A = \sum_i x^{W_i}$, then $As_\lambda = \sum_i s_{\lambda +
    W_i}$.
\end{proof}

In particular, if $A/B = A'/B'$ and $A'$ and $B'$ are Schur positive
symmetric functions, then it is not necessarily the case that $A$ and
$B$ are both Schur positive nor that both $-A$ and $-B$ are both Schur
positive, as the next example shows.

\begin{eg} Let $n=2$. Then $s_4 - s_{3,1}$ is not Schur positive as a
  symmetric function, but it is equal to $s_3(s_4 - s_{3,1}) / s_3 =
  s_7 / s_3$, which is Schur positive in our sense. In this case, $s_4
  - s_{3,1} = x_1^4 + x_2^4$ in the monomial symmetric function
  basis. 
\end{eg}

Furthermore, Proposition~\ref{monomialpositive} is not a necessary
condition. 

\begin{eg} Let $n=2$. Then $s_4 - s_{3,1} - s_{2,2} = x_1^4 -
  x_1^2x_2^2 + x_2^4$ is not positive in the monomial symmetric
  function basis, but the identity
  \[
  \frac{s_5^3(s_4 - s_{3,1} - s_{2,2})}{s_5^3} = \frac{s_{19} +2
    s_{18,1}+2 s_{17,2}+2 s_{16,3}+3 s_{15,4}+4 s_{14,5}+2
    s_{13,6}+s_{11,8}+2 s_{10,9}}{s_5^3}
  \]
  holds, so $x_1^4 - x_1^2x_2^2 + x_2^4 \in \SQ_{\ge 0}$.
\end{eg}

However, there do exist simple necessary conditions. For one thing, a
symmetric function which is Schur positive in our sense must be
positive when we do the substitution $x_1 = x_2 = \cdots = 1$. Less
trivially, if we partially order the weights of a symmetric function
by dominance order ($\lambda$ is said to {\bf dominate} $\mu$ if
$\lambda_1 + \cdots + \lambda_i \ge \mu_1 + \cdots + \mu_i$ for all
$i$), those monomials with maximal weights must have a positive
coefficient.


We mention a criterion for determining if a symmetric function is
equal to a Schur positive fraction. However, this condition seems
difficult to check for large examples. We need some notation. Let $f$
be a polynomial in $d$ variables. Write $f = \sum_I c_I x^I$ and let
${\rm Log}(f) = \{I \mid c_I \ne 0\} \subset \N^d$. We define ${\rm
  conv}({\rm Log}(f))$ to be the convex hull of this set. Given a face
$F$ of this polytope, let $f_F = \sum_{I \in F} c_I x^I$.

\begin{proposition} Let $f$ be a symmetric function in $d$
  variables. Then $f$ is a Schur positive fraction if and only if
  $f_F(r_1, \dots, r_d) > 0$ for all positive real numbers $r_1,
  \dots, r_d$ and all faces $F$ of ${\rm conv}({\rm Log}(f))$.
\end{proposition}

\begin{proof} In \cite[V.6]{handelman}, Handelman shows that for any
  $f$ (not necessarily symmetric), there exists a polynomial $g$ with
  positive coefficients such that $gf$ has positive coefficients if
  and only if $f_F(r_1, \dots, r_d) > 0$ for all positive real numbers
  $r_1, \dots, r_d$ for each face $F$ of ${\rm conv}({\rm Log}(f))$.

  Now let $f$ be a symmetric function which satisfies the hypotheses
  of the theorem. Then we can find some $g$ such that $gf$ has
  positive coefficients. We also have the identity
  \[
  f = \frac{1}{d!} \sum_{\sigma \in \mathfrak{S}_d} \frac{\sigma(g)
    f}{\sigma(g)} = \frac{1}{d!} \frac{h}{\prod_{\sigma \in
      \mathfrak{S}_d} \sigma(g)}
  \]
  for some polynomial $h$. Since $\prod_\sigma \sigma(g)$ is
  symmetric, we conclude that the same is true for $f \prod_\sigma
  \sigma(g)$, and hence $h$ is symmetric. In particular, we expressed
  $f$ as a quotient of two monomial positive symmetric
  functions. Hence we know that $f$ is a quotient of two Schur
  positive symmetric functions by Proposition~\ref{monomialpositive}.
\end{proof}

Now we give the setup for an equivariant version of Boij--S\"oderberg
cones.

First, $C$ lives in a proper subspace of $\SB$. In order to describe
this subspace, let $S_M(d) = \ch(M_d)$ be the {\bf equivariant Hilbert
  function} of $M$, and let $H_M(t) = \sum_{d \ge 0} S_M(d) t^d \in
\SQ_{\ge 0}[[t]]$ be the {\bf equivariant Hilbert series} of $M$. Then
we have
\[
H_{A(-j)}(t) = t^j \sum_{d \ge 0} s_d t^d = \frac{t^j}{(1-x_1t) \cdots
  (1-x_nt)}
\]
as elements of $\SQ_{\ge 0}[[t]]$. We can write an equivariant
resolution ${\bf F}_\bullet$ for $M$:
\[
0 \to \bigoplus_j (A(-j) \otimes \B_{n,j}) \to \cdots \to
\bigoplus_j (A(-j) \otimes \B_{0,j}) \to M \to 0,
\]
where $\B_{i,j}$ is notation for the corresponding representation
with that character. Since this resolution has degree 0 maps, and the
(equivariant) Hilbert function is an additive function on degree 0
exact sequences, we get
\[
H_M(t) = \sum_{i=0}^n (-1)^i H_{F_i}(t) = \frac{\sum_{i=0}^n \sum_j
  (-1)^i \B_{i,j} t^j}{(1-x_1t) \cdots (1-x_nt)}.
\]
Knowing that $M$ is of finite length, $H_M(t)$ must be a polynomial
living in $\SQ_{\ge 0}[t]$, and hence the numerator $\sum_{i=0}^n
\sum_j (-1)^i \B_{i,j} t^j$ is divisible by $(1-x_1t) \cdots
(1-x_nt)$. In particular, we conclude that
\begin{align} \label{equivariantcmequations} \sum_{i=0}^n \sum_j
  (-1)^i \B_{i,j} x_k^{-j} = 0 \quad \text{ for } k=1,\dots,n,
\end{align}
which gives $n$ linearly independent equations. The equivariant Betti
diagrams live in the $\SQ$-subspace defined by these equations. A
simple dimension count allows us to conclude the following.

\begin{proposition} \label{simplicialcase} If $\B(M)$ is an
  equivariant Betti table of a finite length module $M$ which is pure
  in all degrees except possibly one, then $\B(M)$ is a Schur positive
  linear combination of pure Betti tables.
\end{proposition}

\begin{proof} Set $\ul{d}_i = \min \{ j \mid \B(M)_{i,j} \ne 0 \}$ and
  $\ol{d}_i = \max \{ j \mid \B(M)_{i,j} \ne 0 \}$. By our assumption,
  $\ul{d}$ and $\ol{d}$ agree except in at most one coordinate, call
  this coordinate $k$ if it exists. If it does not, then $M$ is pure,
  and there is nothing to show. Otherwise, $\dim \SB_{\ul{d}, \ol{d}}
  = n + 1 + \ol{d}_k - \ul{d}_k$, and the subspace cut out by the
  equivariant Herzog--K\"uhl equations \eqref{equivariantcmequations}
  has codimension $n$. Furthermore, we have $1 + \ol{d}_k - \ul{d}_k$
  linearly independent equivariant pure Betti tables $\B(j)$ coming
  from the degree sequences $d(j)$ for $\ul{d}_k \le j \le \ol{d}_k$
  which are defined by $d(j)_i = \ul{d}_i$ for $i \ne k$ and $d(j)_k =
  j$ otherwise. Hence, $\B(M)$ must be a linear combination $\sum_j
  c_j \B(j)$ of these Betti tables. By comparing which coefficients
  are zero or nonzero, we immediately get $c_j = \B(M)_{k,j} /
  \B(j)_{k,j} \in \SQ_{\ge 0}$.
\end{proof}

\begin{corollary} Every pure equivariant Betti table of a finite
  length module is a Schur positive scalar multiple of a Betti table
  arising from Theorem~\ref{equivariantres}.
\end{corollary}

One could attempt to mimic the proof of Eisenbud and Schreyer to prove
Conjecture~\ref{weakconj}. The main problem seems to be that in the
case of the field $\SQ$, the boundaries of cones are rather
complicated, and so one does not have a nice description of the
exterior facets as in the non-equivariant case.

\subsection{Examples.} \label{equivariantexamples}

We now give some examples which use an equivariant analogue of the
Boij--S\"oderberg algorithm to find a decomposition of Betti
tables. The idea of the algorithm is the same as the one presented in
Section~\ref{bsconesection} except that one uses coefficients in
$\SQ_{\ge 0}$ instead of positive rational numbers. The correctness of
the algorithm in these cases is a consequence of
Proposition~\ref{simplicialcase}. First, we state two propositions
which give some families of identities among Schur polynomials and
then give some examples with actual numbers. In
Example~\ref{countersimplicial}, we will give an example showing that
this algorithm does not always work.

Pick $a > b > 0$. Let $\alpha = (a,b,0)$, $\beta^1 = (a+1,b,0)$, and
$\beta^2 = (a,b+1,0)$. Then the Pieri resolution of the cokernel of
$\beta^1 \oplus \beta^2 \to \alpha$ is
\[
0 \to (a+1,b+1,b+1) \to (a+1,b+1,0) \oplus (a,b+1,b+1) \to (a+1,b,0)
\oplus (a,b+1,0) \to (a,b,0).
\]

\begin{proposition} \label{schurfamily1} We have the following
  equivariant isomorphism of graded Betti tables:
  \begin{align*}
  &(b+1,b+1,0) \otimes 
  \begin{array}{|c|c|c|c|}
    \hline (a,b,0) & (a+1,b,0) \oplus (a,b+1,0) & (a+1,b+1,0) & \\ 
    \hline \vdots & \vdots & \vdots & \vdots \\
    \hline & & (a,b+1,b+1) & (a+1,b+1,b+1) \\
    \hline 
  \end{array} \cong\\
  &(a+1,b+1,0) \otimes 
  \begin{array}{|c|c|c|c|}
    \hline (b,b,0) & (b+1,b,0) & (b+1,b+1,0) & \\ 
    \hline \vdots & \vdots & \vdots & \vdots \\
    \hline & & & (b+1,b+1,b+1) \\
    \hline 
  \end{array}\ \oplus\\
  &(a,b+1,b+1) \otimes 
  \begin{array}{|c|c|c|c|}
    \hline (b,0,0) & (b+1,0,0) & & \\
    \hline \vdots & \vdots & \vdots & \vdots \\
    \hline & & (b+1,b+1,0) & (b+1,b+1,1) \\
    \hline 
  \end{array}\ ,
  \end{align*}
  where the $\vdots$ spans $b-1$ rows of zeroes.

  In particular, we deduce the following three identities:
  \begin{align*}
    (b+1,b+1,0) \otimes (a,b,0) &\cong (a+1,b+1,0) \otimes (b,b,0)\\
    & \quad \oplus (a,b+1,b+1) \otimes (b,0,0)\\
    (b+1,b+1,0) \otimes ((a+1,b,0) \oplus (a,b+1,0)) &\cong
    (a+1,b+1,0) \otimes (b+1,b,0)\\
    &\quad \oplus (a,b+1,b+1) \otimes (b+1,0,0)\\
    (b+1,b+1,0) \otimes (a+1,b+1,b+1) &\cong (a+1,b+1,0) \otimes
    (b+1,b+1,b+1)\\
    &\quad \oplus (a,b+1,b+1) \otimes (b+1,b+1,1).
\end{align*}
\end{proposition}

Let $h_k = s_{(k)}$ be the $k$th complete homogeneous symmetric
function of degree $k$ in $n$ variables. We remark that all three
identities can be proven from the following lemma by expressing both
sides as $3 \times 3$ determinants.

\begin{lemma} \label{determinantlemma} The following identity holds
  \[
  s_\lambda s_\mu = \det(h_{\lambda_i + \mu_{n+1-j} - i +
    j})_{i,j=1}^n
  \]
  if we interpret $h_0 = 1$ and $h_k = 0$ for $k<0$.
\end{lemma}

\begin{proof} See \cite[I.3, Example 8(c)]{macdonald}. \end{proof}

\begin{eg}
  Consider the Pieri resolution with $n=3$, $\alpha = (2,1,0)$,
  $\beta^1 = (3,1,0)$, and $\beta^2 = (2,2,0)$, which is
  \[
  \tiny 0 \to \tableau[scY]{,,|,,||} \to \tableau[scY]{,|,||} \oplus
  \tableau[scY]{,,|,,|} \to \tableau[scY]{,|||} \oplus
  \tableau[scY]{,|,|} \to \tableau[scY]{,||}.
  \]
  The algorithm writes the graded Betti diagram as
  \[
  \left( \begin{array}{cccc} 8 & 21 & 15 & - \\ - & - & 1 &
      3 \end{array} \right) = \frac{5}{2} \left( \begin{array}{cccc} 3
      & 8 & 6 & - \\ - & - & - & 1 \end{array} \right) + \frac{1}{2}
  \left( \begin{array}{cccc} 1 & 2 & - & - \\ - & - & 2 &
      1 \end{array} \right).
  \]
  Multiplying both sides by 6, we can write this equation as 
  \[
  6 \left( \begin{array}{cccc} 8 & 21 & 15 & - \\ - & - & 1 &
      3 \end{array} \right) = 15 \left( \begin{array}{cccc} 3 & 8 & 6
      & - \\ - & - & - & 1 \end{array} \right) +
  \left( \begin{array}{cccc} 3 & 6 & - & - \\ - & - & 6 &
      3 \end{array} \right),
  \]
  which gives a decomposition of equivariant Betti diagrams:
  \[
  {\tiny \tableau[scY]{,|,|}} \otimes
  \begin{array}{|c|c|c|c|}
    \hline 
    {\tiny \tableau[scY]{\bl | ,||\bl}}
    & {\tiny \tableau[scY]{\bl | ,|||\bl\bl}}
    \oplus {\tiny \tableau[scY]{\bl | ,|,|\bl\bl}} & {\tiny
      \tableau[scY]{\bl | ,|,||\bl}}
    & \\
    \hline & & {\tiny \tableau[scY]{\bl |,,|,,|\bl}} & {\tiny
      \tableau[scY]{\bl |,,|,,||\bl}} \\
    \hline
  \end{array} = 
  {\tiny \tableau[scY]{\bl |,|,||\bl}} \otimes \begin{array}{|c|c|c|c|}
    \hline {\tiny \tableau[scY]{\bl | , |\bl}} & {\tiny \tableau[scY]{\bl
        |,||\bl}}
    & {\tiny \tableau[scY]{\bl |,|,|\bl}} & \\
    \hline
    & & & {\tiny \tableau[scY]{\bl |,,|,,|\bl}} \\
    \hline
  \end{array}
  \oplus 
  {\tiny \tableau[scY]{\bl |,,|,,|\bl}} \otimes
  \begin{array}{|c|c|c|c|}
    \hline 
    {\tiny \tableau[scY]{\bl ||\bl}} & {\tiny \tableau[scY]{\bl |||\bl}} & &
    \\ 
    \hline 
    & & {\tiny \tableau[scY]{\bl |,|,|\bl}} & 
    {\tiny \tableau[scY]{\bl |,,|,|\bl}} \\ 
    \hline
  \end{array}\ .
  \]
\end{eg}

We can instead set $\beta^2 = (a,b,1)$ and get

\begin{proposition} Set $c=a-b+1$. We have the following equivariant
  isomorphism of graded Betti tables:
  \begin{align*}
    & (c,c,0) \otimes
    \begin{array}{|c|c|c|c|}
      \hline (a,b,0) & (a+1,b,0) \oplus (a,b,1) & (a+1,b,1) & \\ 
      \hline \vdots & \vdots & \vdots & \vdots \\
      \hline & & (a+1,a+1,0) & (a+1,a+1,1) \\
      \hline 
    \end{array} \cong\\
    &(a+1,b,1) \otimes 
    \begin{array}{|c|c|c|c|}
      \hline (a-b,a-b,0) & (c,a-b,0) & (c,c,0) & \\ 
      \hline \vdots & \vdots & \vdots & \vdots \\
      \hline & & & (c,c,c) \\
      \hline 
  \end{array}\ \oplus\\
  &(a+1,a+1,0) \otimes 
  \begin{array}{|c|c|c|c|}
    \hline (a-b,0,0) & (c,0,0) & & \\
    \hline \vdots & \vdots & \vdots & \vdots \\
    \hline & & (c,c,0) & (c,c,1) \\
    \hline 
  \end{array}\ ,
  \end{align*}
  where the $\vdots$ spans $a-b-1$ rows of zeroes.

  We deduce the following three identities:
  \begin{align*}
    (a-b+1,a-b+1,0) \otimes (a,b,0) &\cong (a+1,b,1) \otimes
    (a-b,a-b,0)\\ 
    &\quad \oplus (a+1,a+1,0) \otimes (a-b,0,0)\\
    (a-b+1,a-b+1,0) \otimes ((a+1,b,0) \oplus (a,b,1)) &\cong
    (a+1,b,1) \otimes (a-b+1,a-b,0)\\
    &\quad \oplus (a+1,a+1,0) \otimes (a-b+1,0,0)\\
    (a-b+1,a-b+1,0) \otimes (a+1,a+1,1) &\cong (a+1,b,1) \otimes
    (a-b+1,a-b+1,a-b+1)\\ 
    &\quad \oplus (a+1,a+1,0) \otimes (a-b+1,a-b+1,1).
  \end{align*}
\end{proposition}

\begin{eg}
  Consider the Pieri resolution with $n=3$, $\alpha = (2,1,0)$,
  $\beta^1 = (3,1,0)$, and $\beta^2 = (2,1,1)$, which is
  \[
  \tiny 0 \to \tableau[scY]{,,|,|,} \to \tableau[scY]{,,|||} \oplus
  \tableau[scY]{,|,|,} \to \tableau[scY]{,|||} \oplus
  \tableau[scY]{,,||} \to \tableau[scY]{,||} .
  \]
  The algorithm writes the graded Betti diagram as
  \[
  \left( \begin{array}{cccc} 8 & 18 & 6 & - \\ - & - & 10 &
      6 \end{array} \right) = \left( \begin{array}{cccc} 3 & 8 & 6 & -
      \\ - & - & - & 1 \end{array} \right) + 5
  \left( \begin{array}{cccc} 1 & 2 & - & - \\ - & - & 2 &
      1 \end{array} \right).
  \]
  Multiplying both sides by 6, we can write this equation as 
  \[
  6 \left( \begin{array}{cccc} 8 & 18 & 6 & - \\ - & - & 10 &
      6 \end{array} \right) = 6 \left( \begin{array}{cccc} 3 & 8 & 6 &
      - \\ - & - & - & 1 \end{array} \right) + 10
  \left( \begin{array}{cccc} 3 & 6 & - & - \\ - & - & 6 &
      3 \end{array} \right),
  \] 
  which gives a decomposition of equivariant Betti diagrams:
  \[
  {\tiny \tableau[scY]{\bl |,|,|\bl}} \otimes \begin{array}{|c|c|c|c|}
    \hline {\tiny \tableau[scY]{\bl |,||\bl}} & {\tiny
      \tableau[scY]{\bl |,|||\bl}} \oplus {\tiny \tableau[scY]{\bl
        |,,||\bl}} & {\tiny \tableau[scY]{\bl |,,|||\bl}}
    & \\
    \hline & & {\tiny \tableau[scY]{\bl |,|,|,|\bl}} &
    {\tiny \tableau[scY]{\bl |,,|,|,|\bl}} \\
    \hline
  \end{array} = 
  {\tiny \tableau[scY]{\bl |,,|||\bl}} \otimes \begin{array}{|c|c|c|c|}
    \hline 
    {\tiny \tableau[scY]{\bl |,|\bl}} &  {\tiny \tableau[scY]{\bl |,||\bl}}
    &  {\tiny \tableau[scY]{\bl |,|,|\bl}} & \\
    \hline 
    & & &  {\tiny \tableau[scY]{\bl |,,|,,|\bl}} \\
    \hline
  \end{array}
  \oplus {\tiny \tableau[scY]{\bl |,|,|,|\bl}} \otimes \begin{array}{|c|c|c|c|}
    \hline
    {\tiny \tableau[scY]{\bl ||\bl}} &
    {\tiny \tableau[scY]{\bl |||\bl}} & & \\
    \hline 
    & & {\tiny \tableau[scY]{\bl |,|,|\bl}} &
    {\tiny \tableau[scY]{\bl |,,|,|\bl}} \\
    \hline
  \end{array}\ .
  \]
\end{eg}

\begin{eg} A more complicated Pieri resolution: let $n=3$, $\alpha =
  (3,1,0)$, $\beta^1 = (4,1,0)$, and $\beta^2 = (3,3,0)$. Then the
  Pieri resolution looks like
  \[
  \tiny 0 \to \tableau[scY]{,,|,,|,||} \to \tableau[scY]{,|,|,||}
  \oplus \tableau[scY]{,,|,,|,} \to \tableau[scY]{ ,||||}
  \oplus \tableau[scY]{,|,|,} \to \tableau[scY]{,|||}.
  \]
  The algorithm writes the graded Betti diagram as
  \[
  \left( \begin{array}{cccc} 15 & 24 & - & - \\ - & 10 & 24 & - \\ - &
      - & 3 & 8 \end{array} \right) = \frac{8}{5}
  \left( \begin{array}{cccc} 8 & 15 & - & - \\ - & - & 10 & - \\ - & -
      & - & 3 \end{array} \right) + \frac{8}{5}
  \left( \begin{array}{cccc} 1 & - & - & - \\ - & 5 & 5 & - \\ - & - &
      - & 1 \end{array} \right) + 3 \left( \begin{array}{cccc} 3 & - &
      - & - \\ - & 10 & - & - \\ - & - & 15 & 8 \end{array} \right).
  \]
  The first step of the algorithm subtracts ${\tiny
    \tableau[scY]{,||||}}$ times the first pure diagram on the right
  hand side from ${\tiny \tableau[scY]{,|||}}$ times the Betti diagram
  on the left hand side, which yields the following table:
  \[
  \begin{array}{|c|c|c|c|} \hline
    \begin{array}{cc} (4,4,0) & (4,3,1) \\ (3,3,2) \end{array} & & & \\ 
    \hline 
    & \begin{array}{ccc} (6,4,0) & (6,3,1) & (5,4,1) \\ (5,3,2) &
      (4,4,2) & (4,3,3) \end{array} 
    & \begin{array}{cc} (6,5,0) & (6,4,1) \\ (5,5,1) &
      (5,4,2) \end{array} & \\ \hline 
    &  & \begin{array}{cc} (6,4,2) & (6,3,3) \\ (5,4,3) \end{array} 
    & \begin{array}{cc} (6,5,2) & (6,4,3) \\ (5,5,3) &
      (5,4,4) \end{array} \\ \hline
  \end{array}\ ,
  \]
  where a collection of partitions in the same entry denotes their
  direct sum. The associated isomorphism of representations (after
  simplifications) is
  \[
   T_0 = T_1 \oplus T_2 \oplus T_3,
  \]
  where
  \[
  T_0 = {\tiny \tableau[scY]{\bl |,|,||\bl}} \otimes {\tiny
    \tableau[scY]{\bl |,|||\bl}} \otimes
  \begin{array}{|c|c|c|c|}
    \hline 
    {\tiny \tableau[scY]{\bl |,|||\bl}} & 
    {\tiny \tableau[scY]{\bl |,||||\bl}} & & \\
    \hline &  {\tiny \tableau[scY]{\bl |,|,|,|\bl}} & 
    {\tiny \tableau[scY]{\bl |,|,|,||\bl}} & \\
    \hline & &  {\tiny \tableau[scY]{\bl |,,|,,|,|\bl}} & 
    {\tiny \tableau[scY]{\bl |,,|,,|,||\bl}} \\
    \hline
  \end{array}\ , \quad
  T_1 =  {\tiny \tableau[scY]{\bl |,|,||\bl}} \otimes {\tiny
    \tableau[scY]{\bl |,||||\bl}} \otimes  
  \begin{array}{|c|c|c|c|}
    \hline 
    {\tiny \tableau[scY]{\bl |,||\bl}} & 
    {\tiny \tableau[scY]{\bl |,|||\bl}} & & \\
    \hline 
    & &  {\tiny \tableau[scY]{\bl |,|,|,|\bl}} & \\
    \hline 
    & &  &  {\tiny \tableau[scY]{\bl |,,|,,|,|\bl}} \\
    \hline
  \end{array}\ ,
  \]
  \[
  T_2 = \left( {\tiny \tableau[scY]{\bl |,|,|,|,|,||\bl}} \oplus {\tiny
      \tableau[scY]{\bl |,,|,|,|,|||\bl}} \oplus {\tiny
      \tableau[scY]{\bl |,,|,|,|,|,|\bl}} \oplus {\tiny
      \tableau[scY]{\bl |,,|,,|,|,||\bl}} \right) \otimes
  \begin{array}{|c|c|c|c|}
    \hline 
    {\tiny \tableau[scY]{\bl |,|\bl}} & & & \\
    \hline 
    &  {\tiny \tableau[scY]{\bl |,|||\bl}} & 
    {\tiny \tableau[scY]{\bl |,|,||\bl}} & \\
    \hline 
    & & & {\tiny \tableau[scY]{\bl |,,|,,||\bl}} \\
    \hline
  \end{array}\ ,
  \]
  \[
  T_3 = {\tiny \tableau[scY]{\bl |,|||\bl}} \otimes {\tiny
    \tableau[scY]{\bl |,,|,,|,|\bl}} \otimes
  \begin{array}{|c|c|c|c|}
    \hline 
    {\tiny \tableau[scY]{\bl ||\bl}} & & & \\
    \hline 
    &  {\tiny \tableau[scY]{\bl ||||\bl}} & & \\
    \hline 
    & &  {\tiny \tableau[scY]{\bl |,|,||\bl}} & 
    {\tiny \tableau[scY]{\bl |,,|,||\bl}} \\
    \hline
  \end{array}\ ,
  \]
  and for $T_3$ we have used the factorization
  \[
  \tiny \tableau[scY]{,,|,,|,|,|||} \oplus \tableau[scY]{,,|,,|,,||||}
  \oplus \tableau[scY]{,,|,,|,,|,||} = \tableau[scY]{,|||} \otimes
  \tableau[scY]{,,|,,|,}\ .
  \]
\end{eg}

\begin{eg} Here is an elaboration of
  Example~\ref{multiplicityexample}. While the module resolved is not
  Cohen--Macaulay, we can throw in an extra relation $\beta^0 =
  (5,1,0)$ to make the cokernel $M$ have finite length. Then the
  equivariant Betti diagram of $M$ is
  \[ 
  T = \begin{array}{|c|c|c|c|} \hline
    (3,1,0) & & & \\
    \hline & (5,1,0) & & \\ \hline & \begin{array}{cc} (4,3,0) &
      (4,2,1) \\ (3,3,1) \end{array} & \begin{array}{ccc} (5,3,0) &
      (5,2,1) & (4,4,0) \\ 2 \cdot (4,3,1) & (4,2,2) &
      (3,3,2) \end{array} & \begin{array}{cc} (5,3,1) & (5,2,2) \\
      (4,4,1) & (4,3,2) \end{array} \\ \hline
  \end{array},
  \]
  and its decomposition is
  \[
  {\tiny \tableau[scY]{\bl ||||\bl}} \otimes T = {\tiny
    \tableau[scY]{\bl |,|||||\bl}} \otimes \begin{array}{|c|c|c|c|}
    \hline 
    {\tiny \tableau[scY]{\bl ||\bl}} & & & \\
    \hline 
    &  {\tiny \tableau[scY]{\bl ||||\bl}} & & \\
    \hline 
    & & {\tiny \tableau[scY]{\bl |,|,||\bl}} &
    {\tiny \tableau[scY]{\bl |,,|,||\bl}} \\
    \hline
  \end{array}
  \oplus \left(\ {\tiny \tableau[scY]{\bl |,|,|,||\bl}} \oplus {\tiny
      \tableau[scY]{\bl |,,|,|||\bl}} \oplus {\tiny \tableau[scY]{\bl
        |,,|,|,|\bl}}\ \right) \otimes \begin{array}{|c|c|c|c|}
    \hline 
    \emptyset & & & \\
    \hline 
    & & & \\
    \hline 
    & {\tiny \tableau[scY]{\bl ||||\bl}} & {\tiny \tableau[scY]{\bl
        |,|||\bl}} &
    {\tiny \tableau[scY]{\bl |,,|||\bl}} \\
    \hline
  \end{array}.
  \]
\end{eg}

\begin{eg} Now we give a decomposition for a non-Cohen--Macaulay
  equivariant module. Let $n=3$, $\alpha = (1,0,0)$ and $\beta =
  (2,1,0)$ so that the Pieri resolution is
  \[
  \tiny 0 \to   \tableau[scY]{,,|,} \to   \tableau[scY]{,|,} \oplus
  \tableau[scY]{,,||} \to \tableau[scY]{,||} \to \tableau[scY]{|}.
  \]
  The decomposition (after some simplifications) is
  \[
  {\tiny \tableau[scY]{\bl |,,||\bl}} \otimes \begin{array}{|c|c|c|c|}
    \hline
    {\tiny \tableau[scY]{\bl ||\bl}} & & & \\
    \hline & {\tiny \tableau[scY]{\bl |,||\bl}} & {\tiny
      \tableau[scY]{\bl |,|,|\bl}} \oplus {\tiny \tableau[scY]{\bl
        |,,||\bl}} & {\tiny
      \tableau[scY]{\bl |,,|,|\bl}} \\
    \hline
  \end{array} = 
  {\tiny \tableau[scY]{\bl |,,|,|\bl}} \otimes \begin{array}{|c|c|c|c|}
    \hline 
    \emptyset & & & \\
    \hline 
    &  {\tiny \tableau[scY]{\bl |||\bl}} &  {\tiny \tableau[scY]{\bl |,||\bl}}
    &  {\tiny \tableau[scY]{\bl |,,||\bl}} \\
    \hline
  \end{array}
  \oplus {\tiny \tableau[scY]{\bl |,,|\bl}} \otimes \begin{array}{|c|c|c|c|}
    \hline 
    {\tiny \tableau[scY]{\bl |||\bl}} & & & \\
    \hline
    & {\tiny \tableau[scY]{\bl |,|,|\bl}} &
    {\tiny \tableau[scY]{\bl |,,|,|\bl}} & \\
    \hline
  \end{array}\ ,
  \]
  where $\emptyset$ denotes the empty partition of 0, i.e.,
  $\Sc_\emptyset V = K$ is the trivial representation of $\GL(V)$.
\end{eg}

\begin{eg} \label{countersimplicial} This example shows that the
  equivariant analogue of the Boij--S\"oderberg algorithm does not
  hold. Let $\dim V = 3$ and $A = \Sym(V)$, let $M$ be the cokernel of
  the Pieri map $A \otimes \Sc_3 V \to A \otimes \Sc_2 V$, and let $N$
  be the cokernel of the Pieri map $A \otimes \Sc_{3,1} V \to A
  \otimes \Sc_{1,1} V$. Then the equivariant Betti table of the
  $A$-module $M \oplus M \oplus N$ is
  \[
  T = 
  \begin{array}{|c|c|c|c|}
    \hline
    {\tiny \tableau[scY]{\bl |||\bl}} \oplus {\tiny \tableau[scY]{\bl
        |||\bl}} \oplus {\tiny \tableau[scY]{\bl |,|\bl}} & {\tiny
      \tableau[scY]{\bl ||||\bl}} \oplus {\tiny \tableau[scY]{\bl
        ||||\bl}} & & \\  
    \hline 
    &  {\tiny \tableau[scY]{\bl |,|||\bl}} & {\tiny \tableau[scY]{\bl
        |,|,||\bl}} & \\ 
    \hline 
    & &  {\tiny \tableau[scY]{\bl |,|,|,|\bl}} \oplus {\tiny
      \tableau[scY]{\bl |,|,|,|\bl}} & {\tiny \tableau[scY]{\bl
        |,,|,|,|\bl}} \oplus {\tiny \tableau[scY]{\bl |,,|,|,|\bl}}
    \oplus {\tiny \tableau[scY]{\bl |,,|,,||\bl}} \\ 
    \hline
  \end{array}\ .
  \]
  The equivariant analogue of the Boij--S\"oderberg algorithm fails on
  this example. The top degree sequence of this diagram is
  $(0,1,3,5)$, whose corresponding pure Betti diagram is
  \[
  T' =
  \begin{array}{|c|c|c|c|}
    \hline
    {\tiny \tableau[scY]{\bl |,||\bl}} & {\tiny \tableau[scY]{\bl
        |,|||\bl}} & & \\  
    \hline
    & & {\tiny \tableau[scY]{\bl |,|,|,|\bl}} & \\
    \hline
    & & & {\tiny \tableau[scY]{\bl |,,|,,|,|\bl}} \\
    \hline
  \end{array}\ .
  \]
  In the non-equivariant case, the algorithm says to subtract the
  largest rational multiple of $T'$ from $T$ which makes the resulting
  table have nonnegative entries. In this case, the rational multiple
  is $4/3$, and corresponds to getting rid of the entry in the first
  row and second column. The equivariant version would say to replace
  $4/3$ by $2s_3 / s_{3,1}$. Alternatively, we can subtract $2s_3T'$
  from $s_{3,1}T$. The second row and third column of the resulting
  table contains $-s_{6,3} + s_{6,2,1} + s_{5,4} + s_{5,2,2} +
  s_{4,4,1} - s_{3,3,3}$, which is not in $\SQ_{\ge 0}$ because its
  most dominant weight $(6,3,0)$ has a negative coefficient. A similar
  phenomena occurs when we replace top degree sequence with bottom
  degree sequence in the Boij--S\"oderberg algorithm.
\end{eg}

\small \noindent Steven V Sam, 
Department of Mathematics, 
Massachusetts Institute of Technology, 
Cambridge, MA 02139, 
{\tt ssam@math.mit.edu}, 
\url{http://math.mit.edu/~ssam/}

\bigskip

\small \filbreak \noindent Jerzy Weyman, 
Department of Mathematics, 
Northeastern University, 
Boston, MA 02115, 
{\tt j.weyman@neu.edu}

\end{document}